\documentclass{amsart}

\usepackage{amsmath}
\usepackage{amsthm,amsfonts,amssymb,color,url}
\usepackage[dvipsnames]{xcolor}
\usepackage[all]{xy}
\usepackage{defs}
\usepackage[colorlinks=true]{hyperref}

\usepackage{tikz}
\usepackage{verbatim}
\usetikzlibrary{arrows}

\def\add{{\rm add}}
\def\wt{\widetilde}
\def\wh{\widehat}
\def\:{{\colon}}
\def\lam{{\lambda}}

\def\veps{{\varepsilon}}
\def\rmlog{{\rm log}}
\def\..{{,\dots,}}
\def\hyp{{\rm hyp}}
\def\hh{{h}}

\begin{document}

\author{Michael Temkin}
\address{Einstein Institute of Mathematics, The Hebrew University of Jerusalem, Giv'at Ram, Jerusalem, 91904, Israel}
\email{temkin@math.huji.ac.il}


\thanks{This work was supported by the European Union Seventh Framework Programme (FP7/2007-2013) under grant agreement 268182.}
\title{Metric uniformization of morphisms of Berkovich curves}

\begin{abstract}
We show that the metric structure of morphisms $f\:Y\to X$ between quasi-smooth compact Berkovich curves over an algebraically closed field admits a finite combinatorial description. In particular, for a large enough skeleton $\Gamma=(\Gamma_Y,\Gamma_X)$ of $f$, the sets $N_{f,\ge n}$ of points of $Y$ of multiplicity at least $n$ in the fiber are radial around $\Gamma_Y$ with the radius changing piecewise monomially along $\Gamma_Y$. In this case, for any interval $l=[z,y]\subset Y$ connecting a point $z$ of type 1 to the skeleton, the restriction $f|_l$ gives rise to a {\em profile} piecewise monomial function $\varphi_y\:[0,1]\to[0,1]$ that depends only on the type 2 point $y\in\Gamma_Y$. In particular, the metric structure of $f$ is determined by $\Gamma$ and the family of the profile functions $\{\varphi_y\}$ with $y\in\Gamma_Y^{(2)}$. We prove that this family is piecewise monomial in $y$ and naturally extends to the whole $Y$. In addition, we extend the classical theory of higher ramification groups to arbitrary real-valued fields and show that $\varphi_y$ coincides with the Herbrand function of $\calH(y)/\calH(f(y))$. This gives a curious geometric interpretation of the Herbrand function, which also applies to non-normal and even inseparable extensions.
\end{abstract}
\maketitle

\section{Introduction}

\subsection{Motivation}
We start with a slightly informal description of the problem and its history that does not require a special knowledge of Berkovich geometry. This paper and its prequel \cite{CTT} study combinatorial structure of morphisms $f\:Y\to X$ between non-archimedean curves over an algebraically closed ground field $k$. We work within the modern framework of Berkovich geometry but the problem is older and can also be asked in the languages of rigid geometry or formal models. It is well known that the combinatorial structure of a smooth and proper non-archimedean curve $X$ is controlled by the closed fiber $\gtX_s$ of any semistable models $\gtX$. In particular, the incidence graph $\Gamma_\gtX$ of $\gtX_s$ naturally embeds into $X$ (in Berkovich setting) and $X\setminus\Gamma_\gtX$ is a disjoint union of open discs. One calls $\Gamma_\gtX$ a skeleton of $X$. Note also that $\Gamma_\gtX$ has a natural structure of a metric genus graph: each vertex is labeled with a number -- the genus of the corresponding irreducible component, and each edge is provided with length -- the modulus of the corresponding analytic annulus.

Before this project the main tool for studying $f$ was the simultaneous semistable reduction theorem: $X$ and $Y$ possess semistable formal models $\gtX$ and $\gtY$ such that $f$ extends to a finite morphism $\gtf\:\gtY\to\gtX$. In particular, $f^{-1}(\Gamma_\gtX)=\Gamma_\gtY$ and $f$ induces a map of metric genus graphs satisfying natural balancing conditions (local constancy of degree and analogs of Riemann-Hurwitz formulas at vertices). Moreover, enlarging the skeleta (and allowing edges of infinite length) one can include the ramification points into $\Gamma_\gtY$ and then $f$ splits to \'etale covers of open discs by open discs on the complements of the skeleta. In the tame case, such covers are split, so, again, all combinatorial structure is encoded by the skeleton $\Gamma_\gtY\to\Gamma_\gtX$ of $f$ or by the closed fiber of $\gtf$. This is worked out by Amini--Baker--Brugall\'e--Rabinoff in \cite{ABBR}.

It was clear that $\Gamma_\gtY\to\Gamma_\gtX$ does not provide an adequate description in the wild case. In particular, one has the following three tightly related indications: (a) \'etale covers of discs can be complicated, (b) the locus $N_{f,>1}\subset Y$, where $f$ is not a local isomorphism, can be a huge set, e.g. the set of all points of a disc of large enough radius, (c) the map $\gtf_s$ does not have to be generically \'etale and in this case it is not really informative, e.g. there are no informative local Riemann-Hurwitz formulas. This indicated that, non-surprisingly, the wild case is substantially more complicated, but it was unclear if there is a finer combinatorics that can explain these phenomena. In fact, there were no positive conjectures in that direction.

The current project started with an observation that one can use the different to explain strange behaviour of double covers $f\:E\to\bfP^1_k$, where $k=\bfC_2$ and $E$ is an elliptic curve with good supersingular reduction. In such case, the point $z\in E$ corresponding to the elliptic component seems to ``appear out of nowhere". For example, if $f$ is given by $y^2=x(x-1)(x-\lam)$ then $f(z)$ lies in the disc around $\sqrt{\lam}$, where no ramification happens (see \cite[\S7.2]{CTT} for a detailed description of such double covers). It was shown in \cite{CTT} that the family of differents viewed as a function $\delta_f\:Y\to[0,1]$ and its restriction onto the skeleta provides a new combinatorial invariant that clarifies the structure of $f$ drastically. In particular, it extends the local Riemann-Hurwitz formula to the case when $\gtf_s$ is inseparable and it provides a ``finite" combinatorial description of the set $N_{f,>1}$ when $f$ is not too wild (the local degrees are not divisible by $p^2$).

The aim of this paper is to remove the latter restriction, that is, to show that there always is a finite datum that completely controls the combinatorial properties of $f$. This will involve a series of piecewise monomial functions on large enough skeleta that can be viewed as finer ramification invariants. Although we will first construct these functions in an elementary self-contained way, one may wonder if they have a natural classical interpretation. In the second part of the paper we will answer this question affirmatively, namely the string of the new invariants can be interpreted as the Herbrand function of the extensions $\calH(y)/\calH(f(y))$ for $y\in\Gamma_\gtY$. A technical obstacle here is that $\calH(y)$ is not discretely valued and to make this precise we will first have to extend the classical ramification theory to the non-discrete case.

\subsubsection{Metric structure and the multiplicity}
Now, let us formulate the goals of this paper in precise terms. Assume that $f\:Y\to X$ is a finite morphism between nice Berkovich curves (see \ref{nicesec}) over an algebraically closed ground field $k$. Note that $Y$ and $X$ possess a natural exponential metric and $f$ is piecewise monomial on intervals $I\subset Y$ with respect to this metric; for example, see \cite[Lemma~3.6.8]{CTT}. Also, there is a natural multiplicity function $n_f\:Y\to\bfN$ associated with $f$, see \S\ref{multsec}.

Our aim is to find a ``finite combinatorial" description of $f$ as a piecewise monomial map between metric graphs. If $f|_I$ is of a constant slope $m$ then $n_f=|m|$ almost everywhere on $I$ (e.g., this follows from \cite[Lemma~3.5.8]{CTT}). In particular, the metric structure of $f$ is described by the multiplicity function $n_f\:Y\to\bfN$, or just by the loci $N_{f,\ge d}$ of points $y\in Y$ of multiplicity at least $d$.

\subsubsection{The prequel}
If $f$ is residually tame, the situation is very simple since $N_{f,\ge 2}$ is contained in a finite graph. However, the sets $N_{f,p^d}$ with $p=\cha(\tilk)$ can be very large in the residually wild case, and their structure was absolutely unclear until very recently; we refer to \cite[Introduction]{CTT} for description of partial results that were known earlier. The aim of this paper and its prequel \cite{CTT} was to find a reasonable combinatorial description of the sets $N_{f,\ge d}$. In \cite{CTT}, we studied the simplest invariant that distinguishes wild ramification -- the different. In particular, we showed that the different function $\delta_f\:Y\to[0,1]$ is piecewise monomial, satisfies a balancing condition at type 2 points and relates the genus of $Y$ to that of $X$. In addition, we showed that the different increases in a standard way outside of any skeleton of $f$, \cite[Theorem~6.1.9]{CTT}, and it completely controls the set $N_{f,p}$ for morphisms of degree $p$, \cite[Theorem~7.1.4]{CTT}. Namely, $N_{f,p}$ is a radial set with center $\Gamma_Y$ and of radius $\delta_f^{1/(p-1)}$.

\subsubsection{This paper}
The first goal of this paper is to prove the radialization theorem that all sets $N_{f,\ge d}$ are radial with respect to a large enough skeleton $\Gamma$ of $f$, and their $\Gamma$-radii are piecewise $|k^\times|$-monomial functions on $\Gamma$. This is the result mentioned in \cite[1.4]{CTT}, and it is proved in the first half of the paper in a pretty elementary and self-contained way. In particular, we use only very basic properties of the different from \cite{CTT}.

Once the radialization theorem is proved, the second goal is to describe the radii of $N_{f,\ge d}$ in terms of classical ramification invariants. The information about the radii is equivalently encoded in the piecewise monomial profile function mentioned in the abstract, and we achieve the second goal by interpreting the profile function as the Herbrand function. In particular, the radii around $y\in\Gamma_Y$ are directly related to the break points of the higher ramification filtration of $\calH(y)/\calH(f(y))$, see Theorem~\ref{radiith}. Note that to make this rigorous, we have also to extend the classical higher ramification theory to real-valued fields with non-discrete valuations.

\subsection{Method and main results}

\subsubsection{The splitting method}
All main results are proved by the same splitting method that reduces the general case to the tame and degree-$p$ cases. For example, we independently introduce and study two types of piecewise monomial functions before comparing them: the profile functions and the Herbrand functions. In both cases we show that

(1) The function is compatible with compositions (of functions or of field extensions, respectively).

(2) The function is trivial in the tame case.

(3) If the degree is $p$ then the function is described by the different $\delta$ as follows: the slopes are 1 and $p$ and the break is at $\delta^{1/(p-1)}$.

The families of such functions $\varphi_y$ or $\varphi_{L/K}$ are completely described by these three conditions. For example, in the case of extensions take the Galois closure $F/K$ of $L/K$. Then $\varphi_{L/K}$ is determined by $\varphi_{F/K}$ and $\varphi_{F/L}$. The two other extensions are Galois, hence split into compositions of tame extensions and extensions of degree $p$, and hence their Herbrand functions are determined by (1)--(3). A similar argument works for $\varphi_y$ after a localization on $X$, see Theorem~\ref{invth}.

\subsubsection{Radialization theorems}
If $l=[z,y]$ is an interval in $Y$ with $z$ of type 1 and $y$ of type 2 then the exponential metric of $Y$ provides a natural homeomorphism of $l$ onto $[0,1]$. We say that a skeleton $\Gamma$ {\em radializes} $f$ if for any interval $l=[z,y]\subset Y$ connecting a point $z$ of type 1 to the skeleton, the restriction $f|_l$, viewed as a function $\varphi_y\:[0,1]\to[0,1]$, depends only on $y$. The collection $\{\varphi_y\}$ is then called the {\em profile} of $f$. It is easy to see that $\Gamma$ radializes $f$ if and only if all sets $N_{f,\ge d}$ are $\Gamma$-radial, see Theorem~\ref{multth}.

Our first main result is that any finite morphism between nice $k$-analytic curves is radialized by a large enough skeleton, see Theorem~\ref{radialth} and Lemma~\ref{enlargelem}(ii). Moreover, we show that if $f$ is either a normal covering, or residually tame, or of degree $p$ then any skeleton of $f$ is radializing, see Theorems~\ref{galoisth} and \ref{simplelem} and Lemma~\ref{tamelem}. Note that we establish the residually tame and degree-$p$ cases first, and the other claims are deduced by local factorization of $f$ into morphisms of these two types.

In addition, we provide examples in Section~\ref{nonradsec} of non-radializing skeletons when the degree of $f$ equals $2p$ and $p^2$.

\subsubsection{The global profile function}\label{globalprofilesec}
To complete the combinatorial description of $f$, one should also show that the $\Gamma$-radii of the sets $N_{f,\ge d}$ depend piecewise monomially on $y\in\Gamma$. Equivalently, one should prove that the profile functions vary piecewise monomially. In fact, we solve a slightly more general problem. Since profile functions are compatible with extensions of skeletons, the radialization theorem implies that to any type 2 point $y$ there is assigned a profile function $\varphi_y$ which possesses the following geometric interpretation: if $l=[z,y]$ is a path starting at a point of type 1 and approaching $y$ from a general direction (i.e. from any but finitely many directions) then $\varphi_y=f|_l$. We prove that this family depends piecewise monomially on $y$ and extends to the set $Y^\hyp$ of all points not of type 1, see Theorem~\ref{extth}.

\subsubsection{Herbrand function}
It is natural to expect that $\varphi_y$ is determined by the ramification theory of the field extension $\calH(y)/\calH(f(y))$. We prove that, indeed, $\varphi_y$ is nothing else but the Herbrand function of $\calH(y)/\calH(f(y))$. Using the splitting method the proof reduces to the tame and degree-$p$ cases, where the comparison is simple. The only obstacle is that the theory of higher ramification was not developed in the non-discrete case, so our main task is to complete this gap. It is known that the meaningful theory of Herbrand functions and upper indexed ramification groups exists only for certain classes of extensions. In the classical situation, one considers monogeneous extensions. In the non-discrete case, one should replace this with a sort of an ``almost" condition. We introduce in Section~\ref{almsec} almost monogeneous extensions and develop for them the theory of upper indexed ramification groups. In addition, we prove that if $x$ is a point of a $k$-analytic curve then any finite extension of $\calH(x)$ is almost monogeneous. On the other hand, we do not know what is the most general class of extensions for which the theory works properly, but see Remark~\ref{logmonrem}(ii) for a possible candidate.

\subsubsection{Structure of the paper}
In Section~\ref{discsec} we study radial morphisms between open discs. This section is very simple and it serves as a preparation to Section~\ref{radsec}, where the radialization theorems are proved. In addition, we extensively study the profile function in Section~\ref{profsec}. In Section~\ref{highsec}, we develop the theory of ramification groups and Herbrand functions $\varphi_{L/K}$ for general real-valued fields, and prove in Theorem~\ref{comparth} that $\varphi_y=\varphi_{\calH(y)/\calH(f(y))}$. In particular, this provides a complete description of the sets $N_{f,\ge d}$ in terms of the Herbrand function, see Theorem~\ref{radiith}. Finally, we explain in the end of Section~\ref{concectsec} how the limit behaviour of $\varphi_y$ when approaching type 1 and 2 points can be naturally described in terms of the {\em logarithmic} Herbrand function of the corresponding extension of valued fields of height two. Since the latter notion is not developed in this paper (and is missing in the literature in the non-discrete case), we only indicate a justification of this description.


\section{Radial morphisms between open discs}\label{discsec}

\subsection{Conventions}

\subsubsection{Ground field}
Throughout the paper $k$ is an algebraically closed complete real-valued field. The valuation can be trivial, though in this case our results are trivial too. By $p\in\{1,2,3,5,7,\dots\}$ we denote the characteristic exponent of $\tilk$. The case of $p=1$ is included for completeness (e.g., see Lemma~\ref{multdisclem} and Remark~\ref{multdiscrem} below), but it will be essentially trivial, and the reader can safely assume that $p>1$.

\subsubsection{$k$-analytic spaces}
We work with strictly $k$-analytic spaces as defined by Berkovich in \cite[Section~1]{berihes}. In particular, $\calM(\calA)$ denotes the spectrum of a Banach $k$-algebra $\calA$ and $\calH(x)$ denotes the completed residue field of a point $x$ on a $k$-analytic space. The valuation of $\calH(x)$ will be denoted of $|\ |_x$, or simply $|\ |$ if $x$ is clear from the context. Also, by $\calAcirc$ the denote the ring of power-bounded elements, by $\calAcirccirc$ the ideal of topologically nilpotent elements and by $\tilcalA=\calAcirc/\calAcirccirc$ the reduction.

\subsubsection{Types of points}
Recall that points on $k$-analytic curves are classified into four types accordingly to $K=\calH(x)$ (e.g., see \cite[Section~3.6]{berihes}): (1) $K=k$, (2) $\tilk\subsetneq\tilK$, (3) $|k^\times|\subsetneq|K^\times|$, (4) the rest.

\subsubsection{Nice compact curves}\label{nicesec}
For shortness, a {\em nice compact curve} means a compact connected separated quasi-smooth strictly $k$-analytic curve throughout this paper. Recall that a curve $X$ is quasi-smooth if it is smooth at all points of type 1, but it may have a boundary. For technical convenience, we include the connectedness assumption, but it can be removed in our main results just by working separately with the connected components.

\subsubsection{Multiplicity}\label{multsec}
Assume that $f\:Y\to X$ is a finite morphism of nice compact curves. Given a point $y\in Y$ with $x=f(y)$ consider the maximal ideals $m_x\subset\calO_{X,x}$ and $m_y\subset\calO_{Y,y}$. Then $m_y=m_x^e\calO_{Y,y}$, where we take $e=1$ if $m_x=0$, and we define the {\em multiplicity} of $f$ at $y$ to be equal to $n_y=e\cdot [\calH(y):\calH(x)]$. The function $n_f\:Y\to\bfN$ sending $y$ to $n_y$ will be called the {\em multiplicity function} associated with $f$. By $N_{f,d}$ or simply $N_d$ we will denote the {\em multiplicity-$d$ locus}, i.e. the set of points $y\in Y$ with $n_f(y)=d$. 

\subsubsection{Open discs}
By an {\em open disc} $E$ we will always mean an open disc whose radius lies in $|k^\times|$, i.e. $E$ is isomorphic to the open unit disc in $\bfA^1_k$.

\subsection{The PL structure}
First, we recall some well known facts related to the piecewise linear structure of open discs. In fact, this will be a piecewise monomial structure since we use the multiplicative notation.

\subsubsection{Piecewise $k^\times$-monomial functions}
As in \cite[Section 1]{bercontr2} or \cite[3.6.3]{CTT}, by a piecewise $|k^\times|$-monomial function on an interval $I\subseteq\bfR_{\ge 0}$ we mean a function $h\:I\to\bfR_{\ge 0}$ such that $I$ is a finite union of closed intervals $I_i$ and $h|_{I_i}$ is a monomial function of the form $c_it^{n_i}$ with $c_i\in|k^\times|$ and $n_i\in\bfZ$; in particular, $h$ is continuous. The integers $n_i$ will be called {\em degrees} or {\em slopes} of $h$.

\subsubsection{Radius function}
Assume that $D$ is an open disc. By a {\em monic coordinate} on $D$ we mean any element $t\in\Gamma(\calO_D)$ that induces an isomorphism of $D$ with the unit open disc. The radius function $r_t(x)=\inf_{c\in k}|t-c|_x$ associated with $t$ is independent of the choice of $t$, and we call it the (monic) radius function of $D$ and denote $r_D$.

\subsubsection{Intervals}\label{intsec}
For any point $x\in D$, by $l_x$ we denote the upward interval starting at $x$, i.e. $l_x$ is semiopen, $x$ is the endpoint of $l_x$ and $l_x$ is relatively compact in $D$. The radius function $r_D$ induces a homeomorphism of $l_x$ onto the interval $[r_D(x),1)$, that we call the {\em radius parametrization} of $l_x$. In particular, if $x$ is a point of type 1
then $l_x$ is identified with the interval $[0,1)$ and given any function $\hh\:D\to\bfR$ we will denote by $\hh_x\:[0,1)\to\bfR$ the restriction of $\hh$ onto $l_x$.

\subsubsection{Restriction of morphisms onto intervals}\label{morintsec}
If $f\:E\to D$ is a morphism of open discs, $y\in E$ and $x=f(y)$ then $f$ maps $l_y$ to $l_x$. In particular, if $y$ is of type 1 then $x$ is of type 1 too and using the radius parameterizations we obtain a map $[0,1)=l_y\to l_x=[0,1)$ that will be denoted $f_y$.

\begin{lem}\label{interlem}
Assume that $f\:E\to D$ is a finite morphism of open discs and $y\in E$ is a point of type 1. Then,

(i) $f_y$ is a piecewise $k^\times$-monomial function that bijectively maps $[0,1)$ onto itself.

(ii) The right logarithmic derivative of $f_y$ coincides with the restriction of the multiplicity function $n_f$ onto $l_y$.
\end{lem}
\begin{proof}
Choose monic parameters of $E$ and $D$ so that $y$ and $f(y)$ become the origins. Then $f$ is given by a series $\phi(t)=\sum_{i=1}^\infty c_it^i$. Let $z$ be the point of $l_y$ of radius $r$. Then $n_f(z)$ is the maximal number $m$ such that $\max_i|c_i|r^i=|c_m|r^m$. In particular, it follows that $n_f$ is an increasing step function.

Let $r<r'<1$ and let $z'$ be the point of $l_y$ of radius $r'$. Taking $r'$ close enough to $r$ we can achieve that $c_mt^m$ is the dominant term of $\phi(t)$ on the interval $[z,z']$, and then $f_y=|c_mt^m|$ on $[z,z']$. This shows that $f_y$ is a strictly increasing piecewise $k^\times$-monomial function and its right logarithmic derivative is equal to $m=n_f$ on $[z,z')$. The bijectivity of $f_y$ follows from the fact that $\lim_{r\to 1}f_y(r)=1$ as $f$ is surjective.
\end{proof}

\subsection{Radial morphisms}

\subsubsection{Radial functions}
We say that a function $\hh\:D\to\bfR$ is {\em radial} if it factors through the radius function, say, $\hh(x)=\varphi(r(x))$ for a real-valued function $\varphi$ on $[0,1)$. We call $\varphi$ the {\em profile} of $\hh$; it will be denoted $\varphi_h$ when needed.

\subsubsection{Radial morphisms}
A morphism $f\:Y\to X$ between open discs is called {\em radial} if the real-valued function $r_X\circ f$ on $Y$ is radial. This happens if and only if there exists a function $\varphi\:[0,1)\to[0,1)$ such that $r_X\circ f=\varphi\circ r_Y$. We call $\varphi=\varphi_f$ the {\em profile} of $f$.

\subsubsection{A criterion for being radial}
A geometric meaning of being radial is described in the following lemma, where the notation $\hh_y$ and $f_y$ is as in \ref{intsec}--\ref{morintsec}.

\begin{lem}\label{critlem}
(i) Assume that $D$ is an open disc and $\hh\:D\to\bfR$ is a function such that for any type 4 point $x$ the restriction of $\hh$ onto $l_x$ is continuous at $x$. Then $\hh$ is radial if and only if the functions $\hh_y$ coincide for all points $y$ of type 1. In this case, $\hh_y$ is the profile function of $\hh$.

(ii) Assume that $f\:Y\to X$ is a morphism of open discs. Then $f$ is radial if and only if the maps $f_y$ coincide for all points $y$ of type 1. In this case, $f_y$ is the profile function of $f$.
\end{lem}
\begin{proof}
Let $D'$ be obtained by removing from $D$ all points of type 4. By the continuity assumption, $h$ is radial if and only if its restriction onto $D'$ is radial. Since $D'$ is covered by the intervals $l_y$ with $y$ a point of type 1, the claim of (i) becomes obvious. The second claim is proved similarly, but this time no continuity assumption is needed because $f$ is automatically continuous.
\end{proof}

\subsubsection{Radial morphisms and the multiplicity function}
It turns out that to check that a morphism is radial it suffices to check that a single integer-valued function, the multiplicity function, is radial.

\begin{lem}\label{multlem}
A morphism between open discs $f\:Y\to X$ is radial if and only if the multiplicity function $n_f$ is radial. In this case, the profile of $n_f$ is the logarithmic derivative from the right of the profile of $f$.
\end{lem}
\begin{proof}
Set $\hh=n_f$ for shortness. Note that the criterion of Lemma~\ref{critlem}(i) applies to $\hh$ because its restriction onto any interval $l_y$ can be discontinuous only at type 2 points by \cite[Lemma~3.6.10]{CTT}. Therefore the lemma follows from Lemma~\ref{critlem} and the fact that for any point $y\in Y$ of type 1 the logarithmic derivative of $f_y$ from the right coincides with $\hh_y$ by Lemma~\ref{interlem}.
\end{proof}

\subsubsection{Composition}
Radial morphisms satisfy the two out of three property with respect to compositions.

\begin{lem}\label{composlem}
Let $f\:Z\to Y$ and $g\:Y\to X$ be finite morphisms between open discs and $h=g\circ f$. If any two morphisms from the triple $f,g,h$ are radial then all three are so. In this case, the profiles are related by the rule $\varphi_h=\varphi_g\circ \varphi_f$.
\end{lem}
\begin{proof}
If $z\in Z$ is a point of type 1, $y=f(z)$ and $x=g(y)$ then $g_y\circ f_z=h_z$. Since the functions $f_z,g_y,h_z$ are invertible by Lemma~\ref{interlem}(i), two of them determine the third one. The assertion now follows from Lemma~\ref{critlem}(ii).
\end{proof}

\subsubsection{Restriction onto smaller discs}
If $X$ is an open disc and $X'\subseteq X$ is an open subdisc of radius $c$ then $r_{X'}=c^{-1} r_X|_{X'}$. This obvious observation implies that the property of being radial is preserved under restrictions onto smaller discs:

\begin{lem}\label{restrdisc}
Let $X$ be an open disc with an open subdisc $X'$ of radius $c$.

(i) If $\hh\:X\to\bfR$ is a radial function of profile $\varphi(t)$ then $\hh|_{X'}$ is a radial function of profile $\varphi(ct)$.

(ii) Assume that $Y$ is an open disc and $f\:X\to Y$ is a radial morphism of profile $\varphi$. Then $Y'=f(X')$ is an open disc and the restriction morphism $f'\:X'\to Y'$ is radial with profile function $\varphi'(t)=a^{-1}\varphi(ct)$, where $a$ is the radius of $Y'$ in $Y$.
\end{lem}
\begin{proof}
The arguments are simple and similar, so we only check (ii). Since $r_Y\circ f=\varphi\circ r_X$, we have that $ar_{Y'}\circ f'=\varphi\circ cr_{X'}$ and hence $r_{Y'}\circ f'=\varphi'\circ r_{X'}$.
\end{proof}

\subsubsection{Multiplicity of radial morphisms}
We will now study local multiplicities of radial morphisms. 

\begin{lem}\label{multdisclem}
Assume that a finite morphism of open discs $f\:Y\to X$ is radial. Then $n_f(z)\in p^\bfN$ for any point $z\in Y$.
\end{lem}
\begin{proof}
Set $d=n_f(z)$. By Lemma~\ref{interlem}, $n_f(y)=d$ for any point $y\in l_z$ close enough to $z$. Choose such a point $y\in l_z\setminus\{z\}$ of type 2, and let $E=\calM(\calA)$ be the closed disc with maximal point $y$. Since $f$ is radial, for any type 1 point $a\in E$ the slope of $|f-f(a)|$ at $y$ in the direction towards $a$ equals $d$. Furthermore, multiplying $f$ by an appropriate $u\in k$ we can assume that $|f|_\calA=1$. Fix a coordinate $x$ on $E$ such that $\calA=k\{x\}$. Then $\kcirc$ parameterizes type 1 points of $E$ and $\tilcalA=\tilk[\tilx]$. If $a\in\kcirc$ then the slope of $|f-f(a)|$ in the direction of the point $x=a$ equals $\ord_\tila(\tilf-\tilf(\tila))$, where $\tilf\in\tilcalA$ is the reduction of $f$. It remains to use the easy fact that $\ord_\tila(\tilf-\tilf(\tila))=d$ for any $\tila\in\tilk$ if and only if $d=p^n$ and $\tilf=b\tilx^d+c$.
\end{proof}

\begin{rem}\label{multdiscrem}
Note we do not have to exclude the case of residual characteristic zero. In this case, $p^\bfN=\{1\}$ and the lemma asserts that a finite morphism $f$ is radial if and only if it is an isomorphism. This illustrates the advantage of using $\expchar(\tilk)$ instead of $\cha(\tilk)$.
\end{rem}

\subsection{Criteria of radiality}

\subsubsection{Finite morphisms of discs}
Let $f\:Y\to X$ be a finite map of open discs. Choose monic coordinates $t$ and $x$, then $f$ is described by sending $x$ to a series $\phi(t)=\sum c_it^i$ with $\max_i|c_i|=1$ and $|c_0|<1$. Choosing the coordinates so that $f$ respects the origins we can also achieve that $c_0=0$. The degree $d=\deg(f)$ is the minimal number with $|c_d|=1$.

\begin{lem}\label{flem}
Let $f\:Y\to X$ be a finite morphism of open or closed discs and $d=\deg(f)$.

(i) If $f$ is \'etale then either $d=1$ or $p>1$ and $d\in p\bfN$.

(ii) If $f$ is an \'etale Galois covering then $d\in p^\bfN$.
\end{lem}
\begin{proof}
If $X$ and $Y$ are open then we can choose monic coordinates $z$ and $t$ so that $f$ is given by $z=\sum_{i=1}^\infty c_it^i$, where $|c_i|<1$ for $i<d$ and $\max_i|c_i|=|c_d|=1$. It follows easily that for any $r<1$ and close enough to 1 the preimage $f^{-1}(X_r)$ of the closed disc $X_r=X\{|z|\le r\}$ is the closed disc $Y_{r^{1/d}}=Y\{|t|\le r^{1/d}\}$. Thus, it suffices to prove the lemma for the finite covering $Y_{r^{1/d}}\to X_r$ of degree $d$, and we assume in the sequel that the discs are closed unit discs.

(i) This time $f$ is given by a series $\phi(t)=\sum_{i=1}^\infty c_it^i\in k\{t\}$ with $|c_i|<1$ for $i>d$ and $\max_i|c_i|=|c_d|=1$. The \'etaleness of $f$ means that $\phi'(t)=\sum_{i=1}^\infty ic_i t^{i-1}$ is invertible on $Y$, and this happens if and only if $|c_1|>|ic_i|$ for all $i>1$. If $d>1$ then $1\ge |c_1|>|dc_d|\ge |d|$, and hence $p>1$ and $d\in p\bfN$.

(ii) Consider a $p$-Sylow subgroup $H$ of $\Gal(Y/X)$. The quotient $Y/H$ is a closed disc by Remark~\ref{imagerem}(i) below (we postpone it to Section~\ref{radsec} for expositional reasons). By part (i), the degree $|G/H|$ of the \'etale morphism $Y/H\to X$ is either 1 or divisible by $p$. The second case is impossible, hence $Y/H=X$ and $d=|H|\in p^\bfN$.
\end{proof}

\subsubsection{Galois coverings}
The degree-$p$ case is easily studied by hand.

\begin{lem}\label{plem}
Let $f\:Y\to X$ be a finite \'etale morphism of discs. If $f$ is of degree $p$ then it is radial and $n_f(y)\in\{1,p\}$ for any $y\in Y$.
\end{lem}
\begin{proof}
We can assume that $f$ is given by a series $\phi(t)=\sum_{i=1}^\infty c_it^i$. Note that $\phi$ satisfies the following condition: (*) $|c_1|>|c_i|$ when $(p,i)=1$, $|c_1|>|pc_p|$, and $|c_p|=1>|c_i|$ for any $i<p$. This condition implies that on the upward interval $l_O$ starting at the origin $O\in Y$, the dominant term of $\phi$ is either $c_1t$ or $c_pt^p$ and the radius of the breaking point satisfies $|c_1|r=|c_p|r^p=r^p$, and so $r=|c_1|^{1/(p-1)}$. Moreover, (*) is invariant under translations of the disc because, by a direct inspection, $\phi(t+b)$ satisfies (*) for any $b\in k$ with $|b|<1$. Therefore, $n_f(z)=1$ if $r(z)<|c_1|^{1/(p-1)}$ and $n_f(z)=p$ otherwise. It remains to use Lemma~\ref{multlem}.
\end{proof}

\begin{cor}\label{Gcor}
If $f\:Y\to X$ is an \'etale Galois covering of degree $d$ of an open disc by an open disc, then the morphism $f$ is radial, $d=p^m$  and $n_f(y)|p^m$ for any $y\in Y$.
\end{cor}
\begin{proof}
Note that $d=p^m$ by Lemma~\ref{flem}(ii) and hence $G=\Gal(Y/X)$ is a $p$-group. Thus, $G$ is solvable and we can factor $f$ into a tower of \'etale coverings of degree $p$. The latter are radial by Lemma~\ref{plem}, hence $f$ is radial by Lemma~\ref{composlem}. The last claim follows from Lemma~\ref{multdisclem}.
\end{proof}

\subsection{Non-radial examples}\label{nonradsec}
After proving affirmative results about radiality, let us discuss the limitations. For this we will construct a few examples of non-radial finite \'etale morphisms $f\:Y\to X$. If $p=1$ then any finite \'etale $f$ is an isomorphism, so we will assume that $p>1$. In this case we will see that $\deg(f)$ can be any number $mp$ with $m>1$.

\subsubsection{A framework}
We will describe a polynomial $\phi(t)$ that defines $f\:Y\to X$. In particular, the derivative $\phi'$ is a unit on $Y$. In all cases, we will exhibit a point $z$ in the interior of the upward interval $l_O$ starting at the origin $O\in Y$ such that $d=n_f(z)\notin p^\bfN$, so $f$ is not radial by Lemma~\ref{multdisclem}.

\subsubsection{Degree $mp$}
Take $\phi=t^{mp}+c_1t$ with $|mp|<|c_1|<1$. Then $\phi'=mpt^{mp-1}+c_1$ and the free term is dominant everywhere on $Y$. So, $\phi'$ is a unit. Note that $n_f$ takes the values 1 and $mp$ on the interval $l_O$. In particular, $f$ is not radial whenever $m\notin p^\bfN$.

In next examples we assume for simplicity that $\cha(k)=p$. The interested reader can easily adjust them to the mixed characteristic case by imposing inequalities analogous to the inequality $|mp|<|c_1|$ above. In addition, we assume that $p\neq 2$.

\subsubsection{Degree $p^2$}
Take $\phi=t^{p^2}+c_{2p}t^{2p}+c_1t$ such that $|c_{2p}|<1$ and $$r_1=|c_1/c_{2p}|^{1/(2p-1)}<r_2=|c_{2p}|^{1/(p^2-2p)}.$$ Then $n_f$ takes the values $1,2p,p^2$ on $l_O$ with break points $r_i$. In particular, $f$ is not radial.

\subsubsection{Split points form a radial set}\label{splitsec}
Finally, choose $\phi=t^{2p^2}+c_pt^p+c_1t$ with $|c_p|<1$ and $$r_1=|c_1/c_{p}|^{1/(p-1)}<|c_{p}|^{1/(p^2-p)}<r_2=|c_{p}|^{1/(2p^2-p)}.$$ In particular, $n_f$ takes the values $1,p,2p^2$ on $l_O$ with breaks at $r_i$, and the value $2p^2$ guarantees that $f$ is not radial. On the other hand, for any $a\in k$ with $|a|<1$ we have that $$\phi(t+a)-\phi(a)=t^{2p^2}+2a^{p^2}t^{p^2}+c_pt^p+c_1t.$$ Since $r_1<|c_{p}/2a^{p^2}|^{1/(p^2-p)}$, the linear term of $\phi(t+a)-\phi(a)$ becomes dominant at the radius $r_1$. Thus, $n_f(z)=1$ if and only if $r(z)<r_1$. In other words, the set of non-split points $N_{f,>1}$ is radial, i.e. consists of all points whose radius exceeds a fixed threshold. Since $f$ is not radial, $n_f$ is not radial and hence some set $N_{f,>d}$ is not radial. In fact, $n_f$ takes the values $1,p,p^2,2p^2$ but already $N_{f,>p}$ is not radial.

\section{Radialization theorems}\label{radsec}

\subsection{Normal coverings}
In this section we fix our terminology about Galois and normal coverings; the material is pretty standard.

\subsubsection{Galois coverings}
Given a finite morphism of nice compact curves $f\:Y\to X$ we will also say that $Y$ or $f$ is a {\em finite covering} of $X$. We say that $f$ is a {\em ramified Galois covering} if the cardinality of $\Aut_X(Y)$ equals the degree of $f$. The word ``ramified'' means that $f$ may have ramification but does not have to. {\em Galois covering} always means \'etale Galois covering. By {\em Galois closure} of a finite covering $Y\to X$ we mean the minimal ramified Galois covering (if it exists) $Z\to X$ that factors through $Y$.

\begin{lem}\label{galcloslem}
Any finite generically \'etale covering of nice compact curves $f\:Y\to X$ possesses a Galois closure $Z\to X$. Moreover, $Z$ can be realized as the normalization of an irreducible component of $(Y/X)^d=Y\times_XY\times_X\dots\times_XY$, the $d$-fold fibred product where $d=\deg(f)$.
\end{lem}
\begin{proof}
Removing a finite set of points of type 1 from $X$ and removing their preimages from $Y$ we obtain a finite \'etale morphism $f'\:Y'\to X'$. In this case it is standard that the Galois closure of $f'$ exists and is realized as a connected component $Z'$ of $(Y'/X')^d$. Let $Z$ be the normalization of the closure of $Z'$ in $(Y/X)^d$. Then $Z$ is a nice compact curve and $g\:Z\to X$ is a finite covering. The fact that $g$ is Galois and minimal follows from the following simple claim: if nice compact curves $Z$ and $T$ are finite coverings of $X$ and $Z'\subseteq Z$, $T'\subseteq T$ are the preimages of $X'$ then any $X$-morphism $g'\:T'\to Z'$ extends uniquely to an $X$-morphism $g\:T\to Z$. To prove this claim, consider the graph $\Gamma'\subset T'\times_XZ'$ of $g'$ and let $\Gamma$ be its closure in $T\times_XZ$. (It is an irreducible component of $T\times_XZ$ since $\Gamma'$ is an irreducible component of $T'\times_XZ'$.) The projection $p\:\Gamma\to T$ restricts to the isomorphism $\Gamma'\to T'$, and using that $T$ is a normal curve we obtain that $p$ is an isomorphism. Thus $T\toisom\Gamma\to Z$ is the extension of $g'$.
\end{proof}

\subsubsection{Radicial coverings}
We say that a finite morphism of nice compact curves $f\:Y\to X$ is {\em radicial} if it is a universal homeomorphism. A typical example is the $n$th power of the geometric Frobenius morphism $F^nX\to X$, which is glued from the morphisms of the form $\calM(\calA)\to\calM(\calA^{p^n})$, where $p=\cha(k)>0$. In fact, they exhaust all radicial morphisms between nice compact curves. Moreover, we have the following lemma:

\begin{lem}\label{radiclem}
Any finite morphism of connected quasi-smooth $k$-analytic curves $Y\to X$ factors uniquely as $Y=F^nZ\to Z\to X$, where $Z\to X$ is a generically \'etale finite covering.
\end{lem}
\begin{proof}
The non-smooth locus of $f\:Y\to X$ is Zariski closed, so either $f$ is generically \'etale and there is nothing to prove or $f$ is nowhere \'etale. In the second case it suffices to prove that $f$ factors uniquely as $Y=FT\to T\to X$, because then induction on the degree of $f$ completes the argument. Let us prove the latter claim. Since it is $G$-local on $X$, we can assume that $X=\calM(\calA)$ is affinoid. Then $Y=\calM(\calB)$ is affinoid too and our claim reduces to showing that $\calA\subseteq\calB^p$. Furthermore, set $K=\Frac(\calA)$ and $L=\Frac(\calB)$. Since $\calB/\calA$ is a finite extension of Dedekind domains it suffices to show that $K\subseteq L^p$.

For any point $y\in Y$ not of type 1 with $x=f(y)$ we have that $\kappa(y)$ is not \'etale over $\kappa(x)$. Thus the extension $\kappa(y)/\kappa(x)$ is inseparable, and since $\kappa(y)$ is a factor of $L\otimes_K\kappa(x)$ we obtain that $L/K$ is inseparable. Thus, it suffices to show that the $p$-rank of $L$ is 1, i.e. $[L:L^p]=p$. By noether normalization, $Y$ is finite over a disc, hence $\calB$ is finite over $\calC=k\{t\}$. Obviously, $\calC$ is of rank $p$ over $\calC^p=k\{x^p\}$, and hence $\Frac(\calC)$ is of $p$-rank 1. It remains to use that $L/\Frac(\calC)$ is finite and the $p$-rank of a field is preserved by finite extensions.
\end{proof}


\subsubsection{Normal coverings}\label{norcovsec}
By a {\em normal covering} of nice compact curves $f\:Y\to X$ we mean a finite morphism which is a composition of a radicial morphism and a ramified Galois covering. Normal closure is defined analogously to Galois closure.

\begin{lem}
Any finite covering of nice compact curves $f\:Y\to X$ possesses a normal closure $Y'\to X$. In fact, $Y'=F^nZ'$, where $Y=F^nZ\to Z\to X$ is the decomposition from Lemma~\ref{radiclem} and $Z'\to X$ is the Galois closure of $Z\to X$.
\end{lem}
\begin{proof}
By definition, $Y'\to X$ is a normal covering. Any normal covering of $X$ dominating $Z$ factors through $Z'$, and using that $Y'=Z'\times_ZY$, we obtain that $Y'\to X$ is the minimal normal covering dominating $Y$.
\end{proof}

\subsection{Skeletons}

\subsubsection{Skeletons of curves}
We adopt from \cite[3.5.1]{CTT} the definition of a skeleton $\Gamma$ of a nice compact curve $X$. In particular, vertices of $\Gamma$ are of types 1 and 2. It is known to experts that the following result is a consequence of the stable reduction, but it is hard to find this in a published literature. We will refer to a book project of Antoine Ducros, and then, for completeness, briefly discuss the main idea of the proof.

\begin{lem}\label{imageskellem}
Assume that $f\:Y\to X$ is a finite morphism between nice compact curves. If $\Gamma_Y$ is a skeleton of $Y$ then $f(\Gamma_Y)$ is a skeleton of $X$.
\end{lem}
\begin{proof}
One should check that any connected component of $X\setminus f(\Gamma_Y)$ is an open disc. This follows from \cite[Lemma~6.2.4]{curvesbook}.
\end{proof}

\begin{rem}\label{imagerem}
(i) In fact, the same argument as in \cite{curvesbook} proves the following slightly stronger fact: if $D$ is a closed (resp. open) disc in $Y$ then either $f(D)=\bfP^1_k=X$ or $f(D)$ is a closed (resp. open) disc too. Here is the main idea. An open disc is a filtered union of closed ones, so it suffices to consider the case when $D$ is closed, and then $E=f(D)$ is a nice compact curve. It is easy to see that $E$ contains neither loops nor positive genus points (i.e. points $x$ of type 2 with $\wHx/\tilk$ of a positive genus) because otherwise their preimage in $D$ would contain a loop or a positive genus point. In addition, the boundary of $E$ is contained in the image of the boundary of $D$, hence it is either empty or a single point. It follows easily from the stable reduction theorem that in the first case $E=\bfP^1_k$ and in the second case $E$ is a disc.

(ii) Here is another approach. The above result is equivalent to the following one: if $X$ is a nice compact curve that possesses a minimal skeleton $\Gamma$ (i.e., $X\neq\bfP^1_k$) and $f\:D\to X$ is a morphism from an open disc, then $f(D)\cap\Gamma=\emptyset$. When $X$ is proper, the latter is recorded in \cite[Theorem~4.5.3]{berbook}. The non-proper case can be reduced to this using that $X$ can be embedded into a nice proper curve $\oX$ by attaching open discs to the boundary.
\end{rem}

\subsubsection{Skeletons of finite coverings}
Let $f\:Y\to X$ be a finite morphism of nice compact curves. We say that a pair of skeletons $(\Gamma_Y,\Gamma_X)$ is {\em compatible} if $f^{-1}(\Gamma_X)=\Gamma_Y$ and $f^{-1}(\Gamma^0_X)=\Gamma^0_Y$, where $\Gamma_X^0$ and $\Gamma_Y^0$ are the sets of vertices. If $f$ is generically \'etale then a skeleton of $f$ is defined in \cite[3.5.9]{CTT} as a compatible pair $\Gamma=(\Gamma_Y,\Gamma_X)$ such that $\Gamma_Y$ contains all ramification points.

If $f$ is not generically \'etale then this definition makes no sense, so we adjust it as follows. Let $Y=F^nZ\to Z\to X$ be the factorization of $f$ with a generically \'etale $g\:Z\to X$. Then by a {\em skeleton} of $f$ we mean any compatible pair of skeletons $(\Gamma_Y,\Gamma_X)$ such that the image of $\Gamma_Y$ in $Z$ contains all ramification points of $g$. The latter condition in fact means that $\Gamma_Y$ contains all points of type 1 where the multiplicity is not locally constant. Also, it is easy to see that $(\Gamma_Y,\Gamma_X)$ is a skeleton of $f$ if and only if $(g^{-1}(\Gamma_X),\Gamma_X)$ is a skeleton of $g$. 

From Lemma~\ref{imageskellem} we immediately obtain the following result.

\begin{lem}\label{translem}
Assume that $f\:Z\to Y$ and $g\:Y\to X$ are finite morphisms of nice compact curves, and $(\Gamma_Z,\Gamma_X)$ is a skeleton of the composition $Z\to X$. Set $\Gamma_Y=f(\Gamma_Z)$. Then $(\Gamma_Z,\Gamma_Y)$ is a skeleton of $f$ and $(\Gamma_Y,\Gamma_X)$ is a skeleton of $g$.
\end{lem}

\subsection{Radial morphisms}

\subsubsection{The retraction $q_\Gamma$}
Assume that $X$ is a nice compact curve with a skeleton $\Gamma$. Since $X\setminus\Gamma$ is a disjoint union of open discs, for any point $x\in X$ there exists a unique interval $l_x=[x,q_\Gamma(x)]$ such that $l_x\cap\Gamma=\{q_\Gamma(x)\}$. (The interval degenerates to a point when $x\in\Gamma$.) Note that $q_\Gamma\:X\to\Gamma$ is the standard retraction of $X$ onto $\Gamma$. If $x\in\Gamma$ then the set $q_\Gamma^{-1}(x)\setminus\{x\}$ is empty if $x$ is of type 3 or 1 and is a disjoint union of open discs if $x$ is of type 2.

\subsubsection{The radius function $r_\Gamma$}
The skeleton $\Gamma$ defines a natural radius function $r_\Gamma\:X\to[0,1]$ as follows. For a point $x\in X$ set $r_\Gamma(x)=\exp(-l)$, where $l$ is the logarithmic length of $l_x$. In particular, $r_\Gamma(x)=0$ if and only if $x$ is a point of type 1, and, more generally, $r_\Gamma$ measures the inverse exponential distance of points of $X$ from $\Gamma$.

\begin{rem}
Any connected component $D$ of $X\setminus\Gamma$ is an open disc and the restriction of $r_\Gamma$ onto $D$ is the usual radius function of $D$.
\end{rem}

\subsubsection{Radial sets}
Given a map $\hh\:\Gamma\to\bfR$ we call $$C(\Gamma,\hh)=\{x\in X|\ r_\Gamma(x)\ge\hh(q_\Gamma(x))\}$$ the {\em radial subset} of $X$ with {\em center} $\Gamma$ and radius $\hh$. Also, we say that $C(\Gamma,\hh)$ is {\em $\Gamma$-radial}.

\subsubsection{Radial functions}
A function $\hh\:X\to\bfR$ is called {\em $\Gamma$-radial} if for any connected component $D$ of $X\setminus\Gamma$, the restriction $\hh|_D$ is radial and its profile $\varphi_D$ depends only on the limit point $q(D)\in\Gamma$ of $D$ in the skeleton, say $\varphi_D=\varphi_{q(D)}$. Note that the profile function $\varphi_q\:[0,1)\to\bfR$ naturally extends to $[0,1]$ by sending $1$ to $h(q)$, and by a slight abuse of notation we will denote the extension by the same letter. The collection $\{\varphi_q\}_{q\in\Gamma^{(2)}}$ is called the {\em profile} of $\hh$, where $\Gamma^{(2)}$ denotes the set of type 2 points of $\Gamma$. Sometimes, it will be convenient to represent the profile as a single function $\varphi^{(2)}\:\Gamma^{(2)}\times[0,1]\to\bfR$. If needed, we will mention $h$ and $\Gamma$ in the notations, e.g. $\varphi_h^{(2)}$.

\subsubsection{Radial morphisms}
Assume that $f\:Y\to X$ is a finite morphism between nice compact curves and $\Gamma=(\Gamma_Y,\Gamma_X)$ is a skeleton of $f$. We say that $f$ is $\Gamma$-radial if for any connected component $E$ of $Y\setminus\Gamma_Y$, the restriction $E\to D=f(E)$ is radial and its profile $\varphi_E$ depends only on the limit point $q(E)\in\Gamma_Y$ of $E$ in the skeleton, say $\varphi_E=\varphi_{q(E)}$. Each function $\varphi_q\:[0,1)\to[0,1)$ is a monotonic bijection, so it extends to the whole $[0,1]$ by continuity. The extension will be denoted by the same letter and the functions $\varphi_q$ give rise to a single profile function $\varphi^{(2)}\:\Gamma_Y^{(2)}\times[0,1]\to\Gamma_X^{(2)}\times[0,1]$. Again, we will sometimes write $\varphi^{(2)}_f$ or $\varphi^{(2)}_\Gamma$.

\subsubsection{Radializing skeletons}
If a morphism $f$ is $\Gamma$-radial then we say that the skeleton $\Gamma$ {\em radializes} $f$. The same terminology will be used for subsets of $Y$ and real-valued functions on $Y$.

\begin{rem}
(i) If $h\:X\to\bfR$ is a $\Gamma$-radial function then one can only extend its profile to a map $\varphi\:\Gamma^{(2)}\times[0,1]\cup\Gamma\times\{1\}\to\bfR$ just by setting $\varphi(q,1)=h(q)$. There is no natural way to define a profile $\varphi_y$ for $y\in\Gamma$ of type 3. The situation with profiles of a radial morphism $f\:Y\to X$ is more interesting. We will later prove that $\varphi_y$ depends on $y$ in a piecewise monomial way, and hence $\varphi^{(2)}$ naturally extends to a map $\varphi\:\Gamma_Y^{(2,3)}\times[0,1]\to\Gamma_X^{(2,3)}\times[0,1]$.

(ii) Recall that radial functions and morphisms on discs were defined in terms of the radius function. In the same fashion, one can define $\Gamma$-radial functions and morphisms in terms of the map $R_\Gamma=(q_\Gamma,r_\Gamma)\:X\to\Gamma\times[0,1]$. Namely, a function $h$ is $\Gamma$-radial if and only if it factors through $R_\Gamma$, and a morphism $f\:Y\to X$ is $(\Gamma_Y,\Gamma_X)$-radial if and only if $r_{\Gamma_X}\circ f$ is a radial function on $Y$, and then $R_{\Gamma_X}\circ f=\phi_f\circ R_{\Gamma_Y}$.
\end{rem}

\subsubsection{Relation to the multiplicity function}
Results of Section~\ref{discsec} easily extend to morphisms between nice compact curves. We start with the results about the multiplicity.

\begin{theor}\label{multth}
Let $f\:Y\to X$ be a finite morphism of nice compact curves and let $\Gamma=(\Gamma_Y,\Gamma_X)$ be a skeleton of $f$. Then the following conditions are equivalent:

(i) The morphism $f$ is $\Gamma$-radial.

(ii) The multiplicity function $n_f$ is $\Gamma_Y$-radial.

(iii) The sets $N_{f,\ge d}:=\{y\in Y|\ n_f(y)\ge d\}$ are $\Gamma_Y$-radial.
\end{theor}
\begin{proof}
Equivalence of (i) and (ii) follows from Lemma~\ref{multlem}. Equivalence of (ii) and (iii) follows from the claim that $n_f$ increases on any interval $l_y$ in $Y$. To check the latter it suffices to consider a finite morphism between open discs, and then the claim was already established in the proof of Lemma~\ref{interlem}.
\end{proof}

\begin{theor}\label{multth2}
Assume that $f\:Y\to X$ is a finite $\Gamma$-radial morphism between nice compact curves, where $\Gamma=(\Gamma_Y,\Gamma_X)$ is a skeleton of $f$. Then $\Gamma_Y$ contains each set $N_{f,d}$ with $d\notin p^\bfN$.
\end{theor}
\begin{proof}
Note that on the complements of $\Gamma_Y$ and $\Gamma_X$ the morphism $f$ splits into a disjoint union of radial finite morphisms between open discs. For these morphisms $n_f$ only accepts the values from $p^\bfN$ by Lemma~\ref{multdisclem}.
\end{proof}

\subsubsection{Composition}
As in the case of discs, radial morphisms are preserved under compositions, but this time we should take the skeletons into account.

\begin{lem}\label{composnicelem}
Let $f\:Z\to Y$ and $g\:Y\to X$ be radial morphisms between nice compact curves with the composition $h\:Z\to X$. Assume that $\Gamma_f=(\Gamma_Z,\Gamma_Y)$, $\Gamma_g=(\Gamma_Y,\Gamma_X)$ and $\Gamma_h=(\Gamma_Z,\Gamma_X)$ are compatible skeletons of $f$, $g$ and $h$, respectively. If two of these skeletons are radializing then all three are radializing and $\varphi^{(2)}_h=\varphi^{(2)}_g\circ\varphi^{(2)}_f$.
\end{lem}
\begin{proof}
This follows from Lemma~\ref{composlem}.
\end{proof}

\subsubsection{Enlarging the skeleton}
Finally, let us show that radial functions and morphisms are preserved by enlarging the skeleton.

\begin{lem}\label{enlargelem}
(i) Assume that $X$ is a nice compact curve with skeletons $\Gamma\subseteq\Gamma'$. If $\hh\:X\to\bfR$ is a $\Gamma$-radial function with profile $\{\varphi_y\}_{y\in\Gamma^{(2)}}$ then $\hh$ is $\Gamma'$-radial with profile $\{\varphi_{y'}\}_{y'\in\Gamma'^{(2)}}$, where $\varphi_{y'}(t)=\varphi_{q_\Gamma(y')}(r_\Gamma(y')t)$.

(ii) Assume that $f\:Y\to X$ is a finite morphism between nice compact curves and $\Delta\subseteq\Delta'$ are skeletons of $X$ whose preimages $\Gamma\subseteq\Gamma'$ in $Y$ are skeletons. If $f$ is $(\Gamma,\Delta)$-radial with profile $\{\varphi_y\}_{y\in\Gamma^{(2)}}$ then $f$ is also $(\Gamma',\Delta')$-radial with profile $\{\varphi_{y'}\}_{y'\in\Gamma'^{(2)}}$, where $\varphi_{y'}(t)=r_{\Delta}(x')^{-1}\varphi_{y}(r_\Gamma(y')t)$ for each $y'\in\Gamma'^{(2)}$ with $y=q_{\Gamma}(y')$ and $x'=f(y')$.
\end{lem}
\begin{proof}
This follows from Lemma~\ref{restrdisc}.
\end{proof}

\subsection{Radialization of morphisms}
Our next aim is to prove that any morphism is radial with respect to a large enough skeleton. In addition, we will see that in certain cases any skeleton is automatically radializing.

\subsubsection{Residually tame coverings}
We say that a finite morphism between nice compact curves $f\:Y\to X$ is {\em residually tame} if for any $y\in Y$ the extension $\calH(y)/\calH(x)$ is tame. (A more restrictive notion of topologically tame morphisms is introduced in \cite[3.2.3]{CTT} by requiring that $n_y$ is invertible in $\tilk$. See \cite[3.2.3]{CTT} for the motivation of this restriction.)

\begin{lem}\label{tamelem}
Assume that $f\:Y\to X$ is a finite residually tame morphism of nice compact curves and $\Gamma=(\Gamma_Y,\Gamma_X)$ is an arbitrary skeleton of $f$. Then $f$ splits outside of $\Gamma$. In particular, $f$ is $\Gamma$-radial and the associated profile $\{\varphi_y\}$ is trivial, i.e. $\varphi_y(t)=t$.
\end{lem}
\begin{proof}
It suffices to prove that if a finite \'etale morphism of open discs $f\:E\to D$ is not an isomorphism then it is not residually tame. By Lemma~\ref{flem}(i) the degree of $f$ is divisible by $p$, hence $f$ is given by a series $\phi(t)=\sum c_it^i$ such that $|c_1|<1=\max_i|c_i|$ and $|c_i|\le |c_1|$ for $i\notin p\bfN$. Choose $r$ close enough to 1 so that all dominant terms of $\phi(t)$ are of the form $c_{pn}t^{pn}$, and let $y$ be the maximal point of the disc around the origin of radius $r$ and $x=f(y)$. Then a direct computation shows that either $\wHy/\wHx$ is inseparable or $|\calH(y)^\times|/|\calH(x)^\times|$ is divisible by $p$. In either case $\calH(y)/\calH(x)$ is not tame.
\end{proof}

\subsubsection{The different}
In \cite{CTT}, a systematic theory of the different function of a morphism $f\:Y\to X$ is developed. In the sequel, we will need a couple of basic properties of the different that we are going to recall. Given a type 2 point $y\in Y$, choose $t\in\calH(y)^\circ$ and $u\in\calH(x)^\circ$ such that $\tilt\notin\wHy^p$ and $\tilu\notin\wHx^p$, and set $\delta_f(y)=\left|\frac{du}{dt}\right|$. We claim that $\delta_f(y)$ is independent of the choice of $t$ and $u$. Indeed, it suffices to show that for any $t'\in\calH(y)^\circ$ with $\tilt'\notin\wHy^p$ we have that $\left|\frac{dt'}{dt}\right|=1$ and similarly for $u$. Since $\Omega_{\wHy/\wHx}=\Omega_{\calH(y)^\circ/\calH(x)^\circ}\otimes_{\calH(y)^\circ}\wHy$, the reduction of $\frac{dt'}{dt}$ equals $\frac{d\tilt'}{d\tilt}$, which is non-zero since both $\{d\tilt\}$ and $\{d\tilt'\}$ are bases of $\Omega_{\wHy/\tilk}$. Thus, $\left|\frac{dt'}{dt}\right|=1$ as claimed. In fact, $t$ and $u$ are tame parameters in the sense of \cite[2.1.2]{CTT} and so $\delta_f(y)$ is the different of the extension $\calH(y)/\calH(x)$ by \cite[Corollary~2.4.6(ii)]{CTT}. We leave it to the reader to check that a similar construction works for a type 3 point $y$, but this time one should take $t$ with $|t|\notin|\calH(y)^\times|^p$ and similarly for $u$, and then $\delta_f(y)=\left|\frac{du}{dt}\right|\cdot|tu^{-1}|$.

\begin{lem}\label{diflem}
Assume that $f\:Y\to X$ is a generically \'etale morphism of nice compact curves and $I\subset Y$ is an interval consisting of type 2 and 3 points. Then $\delta_f$ restricts to a piecewise $|k^\times|$-monomial function on $I$.
\end{lem}
\begin{proof}
This is a particular case of \cite[Corolary~4.1.8]{CTT}.
\end{proof}

\begin{rem}
The proof given in \cite{CTT} is rather straightforward: one covers $I$ by finitely many intervals that possess global tame coordinates. Once such coordinates are available, the claim reduces to piecewise monomiality of $|h|$, where $h=\frac{du}{dt}$. Moreover, \cite[Corolary~4.1.8]{CTT} also deals with the slightly more technical case of points of type 4.
\end{rem}

\subsubsection{Degree-$p$ coverings}
Now, we can extend Lemma~\ref{plem} to nice compact curves.

\begin{theor}\label{simplelem}
If $f\:Y\to X$ is a finite morphism of nice compact curves of degree $p$ then any skeleton $\Gamma=(\Gamma_Y,\Gamma_X)$ is radializing. In addition, for a point $y\in\Gamma^{(2)}$ the profile function $\varphi_y$ is as follows: if $f$ is radicial then $\varphi_y(r)=r^p$, and if $f$ is generically \'etale then $\varphi_y$ has degrees $p$ and $1$ with the breaking point at $r=\delta_f(y)^{1/(p-1)}$. In particular, $n_f(z)\in \{1,p\}$ for any $z\in Y\setminus\Gamma_Y$.
\end{theor}
\begin{proof}
If $f$ is radicial then all profiles are of the form $\varphi_y(r)=r^p$ since $n_f=p$ everywhere on $Y$. So, we can assume that $f$ is generically \'etale. Then the assertion follows from \cite[Theorem~7.1.4]{CTT}, but we prefer to give a self-contained argument for the sake of completeness.

Let $E$ be a connected component of $Y\setminus\Gamma_Y$. Then $E$ and $D=f(E)$ are open discs and $D$ is a connected component of $X\setminus\Gamma_X$. Let $y\in\Gamma_Y$ and $x=f(y)$ be the limit points of $E$ and $D$, respectively. By Lemma~\ref{plem}, the restriction $g\:E\to D$ of $f$ is radial and its profile $\varphi_g$ has degrees 1 and $p$. It remains to check that $\varphi_g$ is as asserted by the theorem. The map $g$ is given by a series $u=\phi(t)=\sum c_it^i$ and we saw in the proof of Lemma~\ref{plem} that $\frac{d\rmlog\varphi_g}{dr}(s)=n_g(s)=1$ if $s<|c_1|^{1/(p-1)}$ and $n_g(s)=p$ otherwise. It remains to observe that for any point $z\in l_O$, where $l_O$ is the upward interval in $E$ starting at the origin $O\in E$, $$\delta_f(z)|ut^{-1}|_z=\left|\frac{du}{dt}\right|_z=|c_1|.$$ So, by Lemma~\ref{diflem} this equality also holds at $y$, i.e. $\delta_f(y)=|c_1tu^{-1}|_y=|c_1|$.
\end{proof}

\subsubsection{Normal coverings}
Normal coverings of nice compact curves can be studied by splitting.

\begin{theor}\label{galoisth}
Assume that $f\:Y\to X$ is a normal covering of nice compact curves. Then any skeleton $\Gamma=(\Gamma_Y,\Gamma_X)$ is radializing.
\end{theor}
\begin{proof}
By Lemmas~\ref{composnicelem} and \ref{translem}, it suffices to prove this assertion in two cases: $f$ is radicial, and $f$ is a ramified Galois covering. The first case is obvious, since $n_f$ is a constant power of $p$. So, we can assume that $f$ is a ramified Galois covering. For any connected component $D$ of $Y\setminus\Gamma_Y$ the induced morphism $f_D\:D\to f(D)$ is a Galois covering whose Galois group $G_D$ can be identified with a subgroup of $G=\Gal(Y/X)$ called the decomposition group of $D$. In particular, Corollary~\ref{Gcor} implies that $f_D$ is radial and $|G_D|=p^m$.

To show that $\Gamma$ is radializing, it suffices to prove that if $D$ and $E$ are connected components of $Y\setminus\Gamma_Y$ with the same limit point in $\Gamma_Y$ then the radial morphisms $f_D$ and $f_E$ have the same profile function. Since $G_D$ is a $p$-group, it can be embedded into a $p$-Sylow subgroup $H$ of $G$. Factor $f$ as $Y\stackrel{a}\to Y/H\stackrel{b}\to X$ and note that $b$ induces an isomorphism $a(D)\toisom f(D)$ since $H$ contains the decomposition group of $D$. Thus, the radial morphisms $f_D$ and $a_D\:D\to a(D)$ are isomorphic, and hence have the same profile.

In the same way, there is a $p$-Sylow subgroup $H'$ containing $G_E$. Since $H$ and $H'$ are conjugate, there is a component $E'$ of $Y\setminus\Gamma_Y$ conjugated to $E$ and such that $H$ contains the decomposition group of $E'$. As above, $f_{E'}\:E'\to f(E')$ and $a_{E'}\:E'\to a(E')$ are isomorphic and hence have the same profile. But the morphism $a$ is radial by Theorem~\ref{simplelem} and the solvability of $H$. Therefore $a_{E'}$ and $a_D$ have the same profile and we obtain that $f_D$ and $f_{E'}$ have the same profile. It remains to note that $f_{E'}$ and $f_E$ are isomorphic via a conjugation, hence their profiles coincide too.
\end{proof}

\subsubsection{General finite coverings}
The above theorem allows to easily construct a radializing skeleton for any finite covering.

\begin{theor}\label{radialth}
Any finite morphism of nice compact curves $f\:Y\to X$ possesses a radializing skeleton. Moreover, if $g\:Z\to X$ is the normal closure of $f$ with the factorization morphism $h\:Z\to Y$ and $(\Gamma_Z,\Gamma_X)$ is any skeleton of $g$, then $(h(\Gamma_Z),\Gamma_X)$ is a radializing skeleton of $f$.
\end{theor}
\begin{proof}
Set $\Gamma_Y=h(\Gamma_Z)$, then $\Gamma=(\Gamma_Y,\Gamma_X)$ is a skeleton of $f$ and $(\Gamma_Z,\Gamma_Y)$ is a skeleton of $h$ by Lemma~\ref{translem}. Since $g$ and $h$ are normal coverings, they are radial with respect to the corresponding skeletons by Theorem~\ref{galoisth}. Therefore, $f$ is $\Gamma$-radial by Lemma~\ref{composnicelem}.
\end{proof}

\subsection{The profile function}\label{profsec}
So far, we only used profiles as a tool in proving radializaton results. Studying fine properties of profiles is the aim of this sections.

\subsubsection{Piecewise monomiality}
First, we prove that the profile function on the radializing skeleton is piecewise monomial. As in \cite[3.3.2]{CTT}, let $X^\hyp$ denote the set of points of $X$ not of type 1. For a skeleton $\Gamma_X$ of $X$ we set $\Gamma_X^\hyp=\Gamma_X\cap X^\hyp$. It is obtained from $\Gamma_X$ by removing all vertices of type 1.

\begin{theor}\label{profth}
If $f\:Y\to X$ is a finite morphism of nice compact curves and $\Gamma=(\Gamma_Y,\Gamma_X)$ is a radializing skeleton, then the profile function $\varphi^{(2)}_\Gamma\:\Gamma^{(2)}_Y\times[0,1]\to\Gamma^{(2)}_X\times[0,1]$ is piecewise $|k^\times|$-monomial, that is, it extends by continuity to a piecewise $|k^\times|$-monomial map $\varphi_\Gamma:\Gamma^{\hyp}_Y\times[0,1]\to\Gamma^{\hyp}_X\times[0,1]$.
\end{theor}
\begin{proof}
We start with two reductions that use the fact that profiles are compatible with compositions by Lemma~\ref{composnicelem}.

First, let $h\:Z\to X$ be the normal closure of $f$, let $\Gamma_Z=h^{-1}(\Gamma_X)$, and let $g\:Z\to Y$ be the morphism $h$ factors through. Then $\varphi^{(2)}_h=\varphi_f^{(2)}\circ \varphi^{(2)}_g$, where the profiles are taken for the corresponding skeleta. Therefore, if the assertion holds for $g$ and $h$, then it also holds for $f$ and its extended profile is determined by $\varphi_h=\varphi_f\circ \varphi_g$. Thus, we are reduced to the case when $f$ is a normal covering.

Second, if $f$ is radicial then the claim is trivial since each $\varphi_y(t)$ is of the form $t^{p^n}$ where $\deg(f)=p^n$. In general, a normal covering factors into the composition of a radicial morphism and a ramified Galois covering. Therefore, it suffices to consider the case when $f$ is a ramified Galois covering, and from now on we make this assumption. We should prove the following claim:

{\em Assume that $e=(u,v)$ is an (open) edge in $\Gamma_Y$ between vertices $u$ and $v$, and let $e'=f(e)$ be its image in $\Gamma_X$. Then the restriction $\varphi_e\:e^{(2)}\times[0,1]\to e'^{(2)}\times[0,1]$ of $\varphi^{(2)}_\Gamma$ is piecewise $|k^\times|$-monomial. In addition, if the endpoint $u$ is of type 2 then the same is true for the interval $[u,v)$.}

We start with two particular cases. Let $A$ be the connected component of $Y\setminus\{u,v\}$ with skeleton $e$. Then both $A$ and $A'=f(A)$ are either open annuli or punctured open discs, and it follows from \cite[Lemma~3.5.8(ii)]{CTT} that the multiplicity of $f$ along $e$ is constant and equals the degree of the finite morphism $f|_A\:A\to A'$. We denote this number $n_e$.

Case 1. {\em Assume that $(n_e,p)=1$.} In this case $\calH(y)/\calH(x)$ is tame for any $y\in e$ and $x=f(y)$. So the profile is trivial on $e\cap Y^{(2)}$ by Lemma~\ref{tamelem} and hence it extends to $e$ trivially. It remains to check that if $u$ is of type 2 then $\varphi_u$ is trivial too. We claim that this is indeed the case since $\calH(u)/\calH(f(u))$ is tame. It is easy to check the latter claim straightforwardly, but let us use a shortcut: since $\delta_f=1$ on $e$, we also have that $\delta_f(u)=1$ by Lemma~\ref{diflem}, and hence $\calH(u)/\calH(f(u))$ is tame by \cite[Lemma~4.2.2(ii)]{CTT}.

Case 2. {\em Assume that $\deg(f)=p$.} In this case, Theorem~\ref{simplelem} tells that each profile function $\varphi_y$ has degrees 1 and $p$ and the breaking point is equal to $\delta_f(y)^{1/(p-1)}$. So, the assertion follows from the fact that $\delta_f(y)$ is piecewise $|k^\times|$-monomial on $e$ by Lemma~\ref{diflem}.

Now, consider the general case. Let $G=\Gal(Y/X)$, let $H\subseteq G$ be the decomposition group of $e$, and let $S$ be a $p$-Sylow subgroup of $H$. Then $f$ splits as $Y=Y_0\to Y_1\to\dots\to Y_n=Y/S\to X$ with the first $n$ morphisms of degree $p$. The claim holds for the morphisms $Y_i\to Y_{i+1}$ by Case 2 and it holds for the morphism $Y/S\to X$ by Case 1. Since profile functions are compatible with compositions by Lemma~\ref{composnicelem}, the claim holds for $f$ as well.
\end{proof}

By continuity, one can also extend the profile function to the points of $\Gamma$ of type 1. We will not study this in details, but only outline the main idea.

\begin{rem}
In the situation of Theorem~\ref{profth} assume that $e=[y,v]$ is an edge of $\Gamma_Y$ and $y$ is of type 1, and let us describe the limit behaviour of $\varphi_f$ at $y$. First, if $f$ splits as $Y\stackrel{F^n}\to Y\to X$ then $\varphi_z(t)=\varphi_{F^n(z)}(t^{p^n})$ for any $z\in(y,v)$. Hence it suffices to study the case when $f$ is generically \'etale. Then the limit behaviour of the different at $y$ was described in \cite[Theorem~4.6.4]{CTT}, hence the same argument as in the proof of Theorem~\ref{profth} shows that:

(1) If $f$ is not wildly ramified at $y$ then $\varphi_z$ is constant for $z\in e$ close enough to $y$. If, moreover, $f$ is not residually wild at $y$ then $\varphi_z$ is trivial (in the sense that $\phi_z(t)=t$) for $z\in e$ close enough to $y$. 

(2) If $f$ is wildly ramified at $y$ then $\varphi_z$ tends to 0 at $y$ (as a function on $[0,1]$). In fact, a finer asymptotic behaviour can be expressed in terms of valued fields of height two, see Section~\ref{type1sec} below.
\end{rem}

\subsubsection{Extension to $X^\hyp$}
Next, we extend the profile function to the whole $X^\hyp$.

\begin{theor}\label{extth}
Assume that $f\:Y\to X$ is a finite morphism of nice compact curves. There exists a unique function $\varphi_f:Y^\hyp\times[0,1]\to X^\hyp\times[0,1]$ such that 

(1) $\varphi_f$ is compatible with profiles of skeletons: for any radializing skeleton $\Gamma=(\Gamma_Y,\Gamma_X)$ of $f$ one has that $\varphi_f|_{\Gamma_Y^\hyp\times[0,1]}=\varphi_{\Gamma}$.

(2) $\varphi_f$ is piecewise $|k^\times|$-monomial: for any interval $I\subset Y^\hyp$ the induced map $I\times[0,1]\to f(I)\times[0,1]$ is piecewise $|k^\times|$-monomial.
\end{theor}
\begin{proof}
Uniqueness is clear since any large enough skeletons is radializing. To prove existence choose any radializing skeleton $(\Gamma_Y,\Gamma_X)$ and consider the corresponding piecewise $|k^\times|$-monomial profile function $\varphi_{\Gamma_Y}$ constructed in Theorem~\ref{profth}. Extend it to a function $\varphi_f$ as follows: for any $y'\in Y$ with $y=q_{\Gamma_Y}(y')$, $r_Y=r_{\Gamma_Y}(y')$ and $r_X=r_{\Gamma_X}(f(y'))$, set $\varphi_f(y,t)=r_X^{-1}\varphi_{\Gamma_Y}(y,r_Yt)$. Clearly, this formula defines a function which is piecewise  $|k^\times|$-monomial on the complement of $\Gamma_Y$. Therefore, $\varphi_f$ satisfies (2). The fact that it also satisfies (1) follows from Lemma~\ref{enlargelem}, in which the same formula relates profiles of two skeletons.
\end{proof}

\subsubsection{A characterization of $\varphi_f$}
As was mentioned in \S\ref{globalprofilesec}, the global profile function $\varphi_f$ can be also characterized geometrically as follows.

\begin{theor}\label{interprof}
Let $f\:Y\to X$ be a finite morphism of nice compact curves. Then $\varphi_f$ is the only piecewise monomial function $\phi\:Y^\hyp\times[0,1]\to X^\hyp\times[0,1]$ such that for any point $y\in Y$ of type 2 the induced function $\phi_y\:[0,1]\to[0,1]$ can be described as follows: if $l=[z,y]$ is a path starting at a point of type 1 and approaching $y$ from a general direction (i.e. from any but finitely many directions) then $\varphi_y=f|_l$.
\end{theor}
\begin{proof}
Uniqueness is clear. Since $\varphi_f$ is piecewise monomial, we should only check its behaviour at a point $y$ of type 2. Choose large enough skeleton $\Gamma=(\Gamma_Y,\Gamma_X)$ of $f$ such that $\Gamma$ is radializing and $y\in\Gamma_Y$. Then $\varphi|_{\Gamma^{(2)}_Y}=\varphi^{(2)}_\Gamma$, and hence the restriction of $\varphi_f$ onto $y\times[0,1]$ is the profile function of $f$ restricted to any open disc of $Y\setminus\Gamma_Y$ attached to $y$. In particular, the asserted property holds for any interval $[z,y]$ such that $[z,y]\cap\Gamma_Y=\{y\}$.
\end{proof}

\subsubsection{Extension of the ground field}
The profile function is compatible with extensions of the ground field, but to formulate this we need first to introduce some notation. Let $K/k$ be an extension of complete algebraically closed real-valued fields. For any nice curve $X$ let $X_K=X\wtimes K$ denote the ground field extension. The fiber of the map $h\:X_K\to X$ over a point $x\in X$ possesses a unique maximal point that we denote $x_K$. In fact, this is true for a general $k$-analytic space, but here everything is very explicit: either $h^{-1}(x)=\{x_K\}$ or $h^{-1}(x)$ is a closed disc with maximal point $x_K$, and the second possibility can only happen (but does not have to) when $x$ is of type 4.

Now we are ready to prove compatibility. In fact, it holds for the section $Y\to Y_K$ that sends $y$ to $y_K$.

\begin{theor}\label{groundext}
Assume that $f\:Y\to X$ is a finite morphism of nice $k$-analytic curves and $g=f_K$ is obtained from $f$ by extending the ground field $k$ to a larger algebraically closed ground field $K$. Then $\varphi_{y}=\varphi_{y_K}$ for any point $y\in Y^\hyp$.
\end{theor}
\begin{proof}
For a point of type 2 this follows from the interpretation in Theorem~\ref{interprof}. If $y\in Y^\hyp$ is arbitrary then we consider a path $[z,y]$ with $z$ of type 2. The map $Y\to Y_K$ sends this path to the path $[z_K,y_K]$ and respects the profiles at points of type 2, which are dense in $[z,y]$. So, the claim follows by continuity.
\end{proof}

\subsubsection{Dependence on the completed residue fields}
Finally, we start a discussion on the dependence of $\varphi_y$ on the extension $\calH(y)/\calH(f(y))$.

\begin{theor}\label{invth}
Let $f\:Y\to X$ be a finite morphism between nice curves, let $y\in Y$ be a type 2 point, $L=\calH(y)$ and $K=\calH(f(y))$. The profile $\varphi_y$ is the invariant of the extension of valued fields $L/K$, say $\varphi_y=\varphi_{L/K}$, which is determined by the following conditions:

(a) $\varphi$ is transitive, i.e. $\varphi_{L/K}=\varphi_{F/K}\circ\varphi_{L/F}$ for any intermediate field $F$.

(b) $\varphi_{L/K}$ is trivial when $L/K$ is tame.

(c) $\varphi_{L/K}(r)=r^p$ if $L/K$ is inseparable of degree $p$.

(d) If $L/K$ is separable of degree $p$ and with different $\delta$ then $\varphi_{L/K}(r)$ has degrees 1 and $p$ and the breaking point is $r=\delta^{1/(p-1)}$.
\end{theor}
\begin{proof}
The profile functions satisfy conditions (b), (c) and (d) by Lemma~\ref{tamelem} and Theorem~\ref{simplelem}. Although the latter theorem addresses only points of type 2, this suffices by continuity since both the profile function and the different function are piecewise monomial, see Theorem~\ref{extth} and \cite[Corolary~4.1.8]{CTT}. It follows from \cite[Theorem~3.4.1]{berihes} that for any $K\subseteq F\subseteq L$ we can shrink $X$ and $Y$ so that $f$ factors through $g\:Y\to Z$ with $\calH(g(y))=F$. Hence condition (a) follows from Lemma~\ref{composnicelem}.

On the other hand, these conditions determine the invariant because any normal extension splits into composition of a tame extension and extensions of degree $p$. In particular, this implies that $\varphi_y$ is determined by $L/K$.
\end{proof}

\begin{rem}
As we saw, there is at most one invariant of extensions of valued fields that satisfies the four conditions from Theorem~\ref{invth}. Existence is not so obvious and requires some restrictions on the fields. In the classical case when the valuations are discrete and the residue fields are perfect, it is well known that the Herbrand function satisfies these conditions. We will show in Section~\ref{highsec} that the theory of higher ramification groups extends to the fields $\calH(y)$, where $y$ is a point on a $k$-analytic curve, and then it will automatically follow that $\varphi_y$ coincides with the Herbrand function of $\calH(y)/\calH(x)$.
\end{rem}

\section{Profile function and higher ramification}\label{highsec}
Unfortunately, many aspects of the theory of valued fields are not developed beyond the discrete case, and it seems that higher ramification is one of them. In this section we try to complete this gap to some extent. We will study the case of real-valued fields with non-discrete valuation. We ignore the discrete-valued case since it is known and requires distinguishing the usual and the logarithmic filtrations and Herbrand functions, see Remark~\ref{ramrem}.

\subsection{Higher ramification groups}

\subsubsection{Notation}
In the sequel, $K$ is a real-valued field, and we assume that the valuation is non-discrete and $p=\cha(\tilK)>0$. 

\subsubsection{Henselian extensions}
We say that a finite extension of valued fields $L/K$ is {\em henselian} if there is a unique extension of the valuation of $K$ to $L$, or, what is equivalent, $\Lcirc/\Kcirc$ is integral. Note that, the valued field $K$ is henselian if and only if any finite extension is henselian. Only henselian extensions will be considered until Section~\ref{concectsec}, so sometimes we will not mention this assumption.

\subsubsection{The prefix ``almost"}
We will use the word ``almost" in the sense of almost mathematics of \cite{Gabber-Ramero}. An almost property $P$ for $K$ means that $P$ ``holds up to something killed" by any element of $\Kcirccirc$. For example, a homomorphism of $\Kcirc$-modules $f\:M\to N$ is an almost isomorphism if both $\Ker(f)$ and $\Coker(f)$ are annihilated by $\Kcirccirc$.

\subsubsection{Ramification groups}
Assume that $L/K$ is a finite henselian Galois extension of valued fields. Define the {\em inertia function} $i_{L/K}\:G\to[0,1]$ and the associated increasing filtration of $G=\Gal(L/K)$ as follows $$i_{L/K}(\sigma)=\sup_{c\in{\Lcirc}}|\sigma(c)-c|,\ \ G_r=\{\sigma\in G|\ i_{L/K}(\sigma)\le r\}.$$ The groups $G_r$ are called
(higher) {\em ramification groups}. Their formation is compatible with subgroups: if $F$ is an intermediate field with the Galois group $H\subseteq G$ then $H_r=G_r\cap H$.

\begin{rem}\label{ramrem}
(i) In the discrete-valued case, one shifts this filtration by $|\pi_K|$. In addition, one considers the logarithmic filtration $G_r^\rmlog$ given by the logarithmic inertia function $i^\rmlog_{L/K}(\sigma)=\sup_{c\in{L^\times}}|\frac{\sigma(c)}{c}-1|$. The two filtrations are pretty close; in fact, it is easy to see that $G_s\subseteq G_s^\rmlog\subseteq G_{s|\pi_K|^{-1}}$ for any $s$. Moreover, these filtrations coincide when $\tilL/\tilK$ is separable.

(ii) In the non-discrete case $i_{L/K}=i^\rmlog_{L/K}$, so there is no need to consider the logarithmic filtration separately.
\end{rem}

\subsubsection{Ramification jumps}
We say that $0\neq r\in|\Lcirc|$ is a {\em jump} of the ramification filtration if there exists $\sigma\in G$ such that $i_{L/K}(\sigma)=r$. This happens if and only if the group $G_{<r}=\cup_{s<r}G_s$ is strictly smaller than $G_r$.

\subsubsection{Herbrand function}
If $L/K$ is as above then the {\em Herbrand function} $\varphi_{L/K}$ is the bijective piecewise monomial function from $[0,1]$ to itself whose breaking points $r_0>r_1>\dots >r_n$ are the jumps of the ramification filtration and whose degrees are described as follows: set $r_{-1}=1$ (so $r_{-1}\ge r_0$) and $r_{n+1}=0$, then the degree on the interval $[r_i,r_{i-1}]$ with $0\le i\le n+1$ is equal to $g_i:=|G_{r_i}|$; for example, $g_0=|G|$ and $g_{n+1}=1$.

\begin{rem}
(i) Since we work with the multiplicative valuations, it is natural to represent Herbrand function as a piecewise monomial function. In the additive setting, the Herbrand function is a piecewise linear function $\varphi_{L/K}^\add$. The two functions are related as additive and multiplicative valuations: $\varphi_{L/K}^\add=-\log(\varphi_{L/K})$.

(ii) In the discrete-valued case, $\varphi_{L/K}$ maps $[0,|\pi_L|^{-1}]$ to $[0,|\pi_K|^{-1}]$ and the slope degree on the interval $[1,|\pi_L|^{-1}]$ equals $e_{L/K}$. Classically, one uses the additive language and normalizes both valuations so that the group of values is $\bfZ$ (e.g., in \cite[Ch. IV, \S3]{locfield}). In this case, one works with the function $\frac{1}{e_{L/K}}\varphi_{L/K}\:[-1,\infty)\to[-1,\infty)$ and the slopes are $g_i/e_{L/K}$.
\end{rem}

\subsubsection{The product formula}\label{prodsec}
It is easy to compute $\varphi_{L/K}$ directly: if $r\in[r_i,r_{i-1}]$ then
$$\varphi_{L/K}(r)=\left(\frac{r}{r_{i-1}}\right)^{g_i}\prod_{0\le j<i}\left(\frac{r_j}{r_{j-1}}\right)^{g_j}=r^{g_i}\prod_{0\le j<i}r_j^{g_j-g_{j+1}},$$
but the following formula will be more useful.

\begin{lem}\label{hlem}
Let $L/K$ be a finite henselian Galois extension of real-valued fields and $r\in[0,1]$. Then $$\varphi_{L/K}(r)=\prod_{\sigma\in G}\max\left(i_{L/K}(\sigma),r\right).$$
\end{lem}
\begin{proof}
It suffices to observe that the righthand side is a piecewise monomial function such that $\varphi_{L/K}(1)=1$ and the degree on $[r_i,r_{i-1}]$ equals $g_i$.
\end{proof}

\subsubsection{The upper indexing}
Using the Herbrand function one introduces a shifted filtration via $G_r=G^s$, where $s=\varphi_{L/K}(r)$. Its jumps $s_i=\varphi_{L/K}(r_i)$ are often called the {\em upper jumps} as opposed to the lower jumps $r_i$.

\begin{rem}
(i) Although the Herbrand function and the upper indexing are defined for any Galois extension $L/K$, they are really meaningful only when some restrictions on $L/K$ are imposed, see Section~\ref{almsec} below. In this case, the Herbrand function is transitive in towers and the upper indexed filtration is compatible with passing to the quotients of $G$. The latter properties are (classically) the main motivation for introducing $\varphi_{L/K}$ and the shifted filtration.

(ii) As explained in \cite[Ch. IV, Remarks~3]{locfield}, the natural group where lower indexes live is $|L^\times|^\bfQ$ while the natural group where upper indexes live is $|K^\times|^\bfQ$, so $\varphi_{L/K}$ can be naturally viewed as a piecewise monomial map between ordered monoids $|\Lcirc|^\bfQ\to|\Kcirc|^\bfQ$. Moreover, this is the only definition making sense for general valued fields, especially of height larger than one. This interpretation also agrees with the fact that the lower indexing is compatible with passing to subgroups of $G$ while (in good cases) the upper indexing is compatible with quotients. In addition, it illustrates the similarity between $\varphi_{L/K}$ and the profile functions.
\end{rem}

\subsection{Almost monogeneous extensions}\label{almsec}
In this section, we introduce a class of extensions $L/K$ for which the ramification theory can be extended further.

\subsubsection{Monogeneous extensions}
A henselian extension of valued fields $L/K$ is called {\em monogeneous} if the extension of integers is so, i.e. $\Lcirc=\Kcirc[x]$. In the classical theory of ramification groups one often assumes that the residue fields are perfect, but one really needs the consequence that the extensions of rings of integers are monogeneous. This is based on the simple observation that for monogeneous extensions, the inertia function can be computed in terms of a generator $x$ as $i_{L/K}(\sigma)=|\sigma(x)-x|$.

\subsubsection{Almost monogeneous extensions}
In the non-discrete case we can extend the above class of extensions as follows: a henselian extension $L/K$ is {\em almost monogeneous} if for any $r<1$ there exists $x_r\in\Lcirc$ and $a_r\in\Kcirc$ such that $a_r\Lcirc\subseteq \Kcirc[x_r]$ and $r\le|a_r|$. An element $x_r$ will be called an {\em $r$-generator}. Note that automatically $K(x_r)=L$ and $L/K$ is finite.

\begin{rem}\label{logmonrem}
(i) We will see that higher ramification theory works fine for the class of almost monogeneous extensions. However, it is not clear if this class satisfies reasonable functoriality properties. For example, I do not know if it is closed under subextensions.

(ii) One may wonder what is the largest class of extensions to which the higher ramification theory extends. Perhaps, these are extensions $L/K$ such that the module $\Omega_{L/K}$ is almost cyclic. One can show that this class is closed under passing to subextensions and contains all almost monogeneous extensions and all separable extensions of degree $p$.
\end{rem}

\subsubsection{Almost monogeneous valued fields}
We say that a real-valued field $K$ is {\em almost monogeneous} (resp. {\em monogeneous}) if any finite henselian Galois extension $L/K$ is so. Some examples are listed below. The main conclusion is that if $y$ is a point on a $k$-analytic curve (and $k$ is algebraically closed) then $\calH(y)$ is almost monogeneous.

\begin{exam}\label{monoexam}
Recall that a henselian valued field $K$ is {\em stable} if $e_{L/K}f_{L/K}=[L:K]$ for any finite extension $L/K$. For a $k$-analytic curve with a point $x$ of type 2 or 3 the field $\calH(X)$ is stable (e.g., see \cite[Theorem~6.3.1(iii)]{temst}).

(i) If $f_{L/K}=[L:K]$ and $\tilL/\tilK$ is generated by a single element then $L/K$ is monogeneous and an element $x\in\Lcirc$ is a generator of $\Lcirc$
if and only if $\tilx$ generates $\tilL$ over $\tilK$. In particular, if $K$ is stable, $\tilK$ has $p$-rank 1 and $|K^\times|$ is divisible, then $K$ is monogeneous. This includes the fields $\calH(y)$, where $y$ is a type 2 point on a $k$-analytic curve.

(ii) If $e_{L/K}$ coincides with $n=[L:K]$ and $H=|L^\times|/|K^\times|$ is cyclic then $L/K$ is almost monogeneous and any element $x\in\Lcirc$, such that $|x|\ge r$ and $|x|$ generates $H$, is an $r^{n-1}$-generator. In particular, if $K$ is stable and not discrete-valued, $\tilK$ is algebraically closed and any finite subgroup of $|K^\times|^\bfQ/|K^\times|$ is cyclic, then $K$ is almost monogeneous. This includes the fields $\calH(y)$, where $y$ is a type 3 point on a $k$-analytic curve.

(iii) If $K=\calH(y)$ and $y$ is of type 4 then any finite extension $L/K$ is almost monogeneous. To prove this one should use that by stable reduction, $L=\calH(z)$ with $z$ a point of a disc. Hence $L=\wh{k(t)}$ and by a direct computation one can show that the elements $t_i=t-a_i$ with $a_i\in k$ and $|t-a_i|$ tending to $\inf_{c\in k}|t-c|$ provide a series of $r_i$-generators of $L/K$ with $r_i$ tending to 1. Perhaps the easiest way to do this is to extend the ground field from $k$ to $K$: since the completion of $\Frac(K\otimes_kK)$ is of type 2 or 3 over $K$, the claim reduces to one of the cases described in (i) and (ii).
\end{exam}

\subsubsection{Bounds on $i_{L/K}$}
Our motivation to introduce $r$-generators is that they provide the following control on the inertia function.

\begin{lem}\label{ilem}
If $L/K$ is a finite henselian Galois extension of real-valued fields with an $r$-generator $x$ and $\sigma\in\Gal(L/K)$ then $$|\sigma(x)-x|\le i_{L/K}(\sigma)\le r^{-1}|\sigma(x)-x|.$$
\end{lem}
\begin{proof}
Only the right inequality needs a proof. Since $|x|\le 1$, we have that $$|x^i\sigma(x^{n-i})-x^{i-1}\sigma(x^{n-i+1})|\le|\sigma(x)-x|$$ for any $n\ge 1$ and $1\le i\le n$. It follows that $|\sigma(x^n)-x^n|\le|\sigma(x)-x|$, and hence any $y\in\Kcirc[x]$ satisfies $|\sigma(y)-y|\le|\sigma(x)-x|$. Since $a_r\Lcirc\subseteq\Kcirc[x]$ for some $a_r\in\Kcirc$ with $r\le|a_r|$ we obtain that $$r|\sigma(z)-z|\le|a_r|\cdot|\sigma(z)-z|=|\sigma(a_rz)-a_rz|\le|\sigma(x)-x|$$ for any $z\in\Lcirc$. Thus, $ri_{L/K}(\sigma)\le|\sigma(x)-x|$, as required.
\end{proof}

\subsubsection{The key lemma}
Now, we are going to establish a key result that relates $i_{L/K}$ to $i_{F/K}$ for $L/F/K$. In fact, this is the only computation where we directly (via Lemma~\ref{ilem}) use that $L/K$ is almost monogeneous; all other results will use this assumption via the key lemma. The standard proof in the classical case is due to J. Tate; it is short but rather tricky, see \cite[Ch. IV, Prop. 3]{locfield}. Tate's proof extends to almost monogeneous extensions straightforwardly.

\begin{lem}\label{keylem}
Assume that $L/K$ is an almost monogeneous Galois extension of real-valued fields and $F$ is the invariant field of a normal subgroup $H\subseteq G=\Gal(L/K)$. Then $i_{F/K}(\sigma)=\prod_{\tau\mapsto\sigma}i_{L/K}(\tau)$ for any $\sigma\in G/H$.
\end{lem}
\begin{proof}
Let $r<1$ and fix an $r$-generator $u\in\Lcirc$ with minimal polynomial over $F$ $$f(t)=\sum_{j=0}^da_jt^j=\prod_{\tau\in H}(t-\tau(u)).$$ Consider the $\sigma$-translate $$f^\sigma(t)=\sum_{j=0}^d\sigma(a_j)t^j=\prod_{\tau\mapsto\sigma}(t-\tau(u))$$ of $f(t)$. Substituting $t=u$ and using Lemma~\ref{ilem} we obtain that
\begin{equation}\label{eq1}
|f^\sigma(u)|\le\prod_{\tau\mapsto\sigma}i_{L/K}(\tau)\le r^{-d}|f^\sigma(u)|.
\end{equation}
On the other hand, since $f(t)\in\Fcirc[t]$ and $u\in\Lcirc$ we obtain that
$$
|f^\sigma(u)|=|f^\sigma(u)-f(u)|\le \max_j |\sigma(a_j)-a_j|\cdot|u|^j\le i_{F/K}(\sigma).
$$
Combining this with the right side of (\ref{eq1}) and using that $r$ can be arbitrarily close to 1, we obtain that $i_{F/K}(\sigma)\ge\prod_{\tau\mapsto\sigma}i_{L/K}(\tau)$.

Let us prove the opposite inequality. We should check that $$|v-\sigma(v)|\le\prod_{\tau\mapsto\sigma}i_{L/K}(\tau)$$ for any $v\in\Fcirc$. If the inequality fails for some $v$ then using that $K$ is not discrete-valued, we can find $a\in\Kcirccirc$ such that the inequality fails also for the element $av\in\Fcirccirc$. So, it suffices to prove the inequality only for $v\in\Fcirccirc$. For such $v$ we can enlarge $r$ and adjust $u$ so that $v\in\Kcirc[u]$, say $v=h(u)$ for $h(t)\in\Kcirc[t]$. Then $u$ is a root of $h(t)-v\in\Fcirc[t]$ and therefore $h(t)-v=f(t)g(t)$ in $\Fcirc[t]$. Since $h^\sigma=h$, we have that $h(t)-\sigma(v)=f^\sigma(t)g^\sigma(t)$ and substituting $t=u$ gives $v-\sigma(v)=f^\sigma(u)g^\sigma(u)$. Therefore, $|v-\sigma(v)|\le|f^\sigma(u)|$, and using the left side of (\ref{eq1}) we obtain that $|v-\sigma(v)|\le\prod_{\tau\mapsto\sigma}i_{L/K}(\tau)$, as required.
\end{proof}

\subsection{Herbrand's theorem}

\subsubsection{Index transition}
To compare $i_{L/K}$ and $i_{F/K}$, for any $\sigma\in G/H$ set $$j_{L/F/K}(\sigma)=\min_{\tau\mapsto\sigma}(i_{L/K}(\tau)).$$

\begin{lem}\label{jlem}
Let $L/K$ be as in Lemma~\ref{keylem} and let $F$ be the invariant field of a normal subgroup $H\subseteq G$. Then $i_{F/K}(\sigma)=\varphi_{L/F}(j_{L/F/K}(\sigma))$.
\end{lem}
\begin{proof}
Choose $\tau$ above $\sigma$ such that $j_{L/F/K}(\sigma)=i_{L/K}(\tau)$. We claim that 
$$i_{L/K}(\tau\lam)=\max\left(i_{L/K}(\lam),j_{L/F/K}(\sigma)\right)$$ for any $\lam\in H$. 
Since $i_{L/K}(\tau\lam)\ge j_{L/F/K}(\sigma)=i_{L/K}(\tau)$, we should show that if $$i_{L/K}(\tau\lam)>i_{L/K}(\tau)\ \ \  {\rm or}\ \ \  i_{L/K}(\lam)>i_{L/K}(\tau)$$ then $i_{L/K}(\tau\lam)=i_{L/K}(\lam)$. To prove this observe that if  $$|\tau\lambda(c)-c|>|\tau\lambda(c)-\lambda(c)|\ \ \ {\rm or}\ \ \  |\lambda(c)-c|>|\tau\lambda(c)-\lambda(c)|$$ then $|\tau\lambda(c)-c|=|\lambda(c)-c|$ and pass to the supremum over $c\in\Lcirc$.


Now, the key lemma \ref{keylem} yields that $$i_{F/K}(\sigma)=\prod_{\tau\mapsto\sigma}i_{L/K}(\tau)=\prod_{\lam\in H}\max(i_{L/K}(\lam),j_{L/F/K}(\sigma)),$$
and by Lemma~\ref{hlem} the righthand side equals $\varphi_{L/F}(j_{L/F/K}(\sigma))$.
\end{proof}

\begin{cor}\label{jcor}
Let $L/F/K$ be as above. If $s=\varphi_{L/F}(r)$ then $G_r H/H=(G/H)_s$.
\end{cor}
\begin{proof}
Note that $\sigma\in G_r H/H$ if and only if $j_{L/F/K}(\sigma)\le r$. By Lemma~\ref{jlem}, the latter happens if and only if $i_{F/K}(\sigma)\le s$, and this happens if and only if $\sigma\in(G/H)_s$.
\end{proof}

\subsubsection{Herbrand's theorem}
Now, we can prove our main result about the Herbrand function and the upper indexed filtration.

\begin{theor}\label{Herth}
Assume that $L/K$ is an almost monogeneous Galois extension of real-valued fields and $F$ is the invariant field of a normal subgroup $H\subseteq G=\Gal(L/K)$. Then,

(i) $\varphi_{L/K}=\varphi_{F/K}\circ\varphi_{L/F}$,

(ii) $(G/H)^s=G^sH/H$ for any $s\in[0,1]$.
\end{theor}
\begin{proof}
(i) Let $r\in[0,1]$ and $s=\varphi_{L/F}(r)$. Both $\varphi_{L/K}$ and $\varphi_{F/K}\circ\varphi_{L/F}$ are piecewise monomial functions which are equal to 1 at 1, hence it suffices to check that their degree at $r\in[0,1]$ coincide (at the break points we take the degree from the left). The degree of the composite at $r$ is equal to $$\deg\varphi_{F/K}(s)\deg\varphi_{L/F}(r)=|(G/H)_s|\cdot |H_r|.$$ By Corollary~\ref{jcor}, the latter is equal to $|G_r|$, which is the degree of $\varphi_{L/K}$ at $r$.

(ii) Choose $r$ with $s=\varphi_{F/K}(r)$. Then $(G/H)^s=(G/H)_r=G_tH/H$ by Corollary~\ref{jcor}, where $r=\varphi_{L/F}(t)$. It remains to note that $\varphi_{L/K}(t)=s$ by part (i) and hence $G_t=G^s$.
\end{proof}

\begin{rem}
In the classical situation, the second part of Theorem~\ref{Herth} is called Herbrand's theorem. However, the first part is, perhaps, even more important. In a sense, it shows that the Herbrand function is a reasonable invariant of an almost monogeneous extension $L/K$.
\end{rem}

\subsubsection{Herbrand function for non-normal extensions}
Part (i) of Herbrand's theorem allows us to extend the definition of the Herbrand function to non-normal separable extensions $F/K$ such that the Galois closure $L/K$ of $F/K$ is almost monogeneous. Indeed, in this case the Galois extension $L/F$ is also almost monogeneous, and we define $\varphi_{F/K}$ to be the piecewise monomial function that satisfies $\varphi_{L/K}=\varphi_{F/K}\circ\varphi_{L/F}$. In particular, in view of Example~\ref{monoexam}, the Herbrand function is defined for any finite separable extension $L/\calH(x)$, where $x$ is a point on a $k$-analytic curve. It is easy to see that in this setting Theorem~\ref{Herth}(i) extends to non-normal extensions.

\subsection{Other properties of $\varphi_{L/K}$}

\subsubsection{Tame extensions}
Since the valuation of $K$ is not discrete, if $L/K$ is tame then any $\sigma\in G$ satisfies $i_{L/K}(\sigma)=1$. Indeed, it follows from the standard theory of tame extensions that either $\sigma$ acts non-trivially on $\tilL/\tilK$ or it acts on an element $x\in\Lcirc$ via $\mu_n$ with $n$ invertible in $\tilK$. In the latter case we can multiply $x$ by an element of $K$ making $|x|$ arbitrarily close to 1. It follows that $i_{L/K}(\sigma)=1$ in either case, and we obtain the following result.

\begin{lem}\label{tameherlem}
If $L/K$ is a finite tame henselian Galois extension then $G_s=1$ for any $s<1$. So, the ramification filtration is trivial, 1 is the only jump point, and the Herbrand function is the identity.
\end{lem}

Since tameness of an extension is preserved by passing to the Galois closure and the Herbrand function is multiplicative in towers by Theorem~\ref{Herth}(i), we also obtain:

\begin{cor}\label{tamehercor}
If $L/K$ is a finite tame henselian extension, then $\phi_{L/K}(t)=t$.
\end{cor}

\subsubsection{The different}
In general, the different $\delta_{L/K}$ of a finite separable extension of real-valued fields $L/K$ can be defined as the zeroth Fitting ideal of the module of differentials $\Omega_{\Lcirc/\Kcirc}$. This requires some care since $\Omega_{\Lcirc/\Kcirc}$ is only almost finitely generated, see \cite[Section~VI.6.3]{Gabber-Ramero}. However, in the differential rank one case this simplifies and one can use the usual definition, namely $$\delta_{L/K}=|\Ann(\Omega_{\Lcirc/\Kcirc})|$$ whenever $\Omega=\Omega_{\Lcirc/\Kcirc}$ is {\em almost cyclic}, i.e. for any $r\in|\Lcirccirc|$ there exists $a_r\in\Omega$ such that $\Omega/a_r\Omega$ is killed by any $\pi\in L$ with $|\pi|\le r$. Note that in this case $\delta_{L/K}$ is the limit of the annihilators of elements $a_r$.

\begin{lem}\label{difalmlem}
Assume that $L/K$ is a finite separable extension of real-valued fields such that $\Lcirc$ is the filtered union of monogeneous subrings $A_i=\Kcirc[x_i]$. Then $\delta_{L/K}=\lim_i |f'_i(x_i)|$, where $f_i$ is the minimal polynomial of $x_i$ over $K$.
\end{lem}
\begin{proof}
K\"ahler differentials are compatible with filtered colimits, hence $\Omega_{\Lcirc/\Kcirc}$ is the filtered colimit of $\Omega_{A_i/\Kcirc}=\Kcirc dx_i/\Kcirc f'_i(x_i) dx_i$ and we obtain that $$\delta_{L/K}=\lim_i|\Ann(\Omega_{A_i/\Kcirc})|=\lim_i|f'_i(x_i)|.$$
\end{proof}

\subsubsection{Upper jumps and the different}
In the classical discrete-valued case, one can compute the different in terms of the ramification jumps, see \cite[Ch. IV, Prop. 4]{locfield}. This extends to our case as follows.

\begin{theor}\label{logdifth}
Assume that $L/K$ is an almost monogeneous Galois extension of real-valued fields and the valuation of $K$ is not discrete. Let $r_0>r_1>\dots>r_n$ be the jumps of the ramification filtration, $g_i=|G_{r_i}|$ and $g_{n+1}=1$. Then $\delta_{L/K}=\prod_{i=0}^nr_i^{g_i-g_{i+1}}$, in particular, $\delta_{L/K}$ is the coefficient of the linear part of $\varphi_{L/K}$, i.e. $\varphi_{L/K}(t)=\delta_{L/K} t$ on the interval $[0,r_n]$.
\end{theor}
\begin{proof}
It suffices to establish the formula for $\delta_{L/K}$ since the second claim then follows from the product formula in \ref{prodsec}. Given $s\in|\Lcirccirc|$ choose an $s$-generator $x_s\in\Lcirc$ and let $f_s(t)$ be its minimal polynomial over $K$. Then $\Lcirc$ is the filtered union of its subrings $\Kcirc[x_s]$, and so $$\delta_{L/K}=\lim_{s\to 1}|f'_s(x_s)|=\lim_{s\to 1}\prod_{\sigma\neq 1}|x_s-\sigma(x_s)|$$ by Lemma~\ref{difalmlem}. By Lemma~\ref{ilem}, the latter limit is equal to $\prod_{\sigma\neq 1}i_{L/K}(\sigma)$. It remains to note that $i_{L/K}(\sigma)=r_i$ if and only if $\sigma\in G_{r_i}\setminus G_{r_{i+1}}$.
\end{proof}

Using that the different and the Herbrand function are multiplicative in towers, we also obtain:

\begin{cor}\label{logdifcor}
Assume that $L/K$ is an extension of real valued fields such that the Galois closure of $L/K$ is almost monogeneous. Then $\delta_{L/K}$ is the coefficient of the linear part of $\varphi_{L/K}$.
\end{cor}

\subsubsection{Extensions of degree $p$}
For extensions of degree $p$ there is a single break point $r_0$. Using the above corollary we see that $\varphi_{L/K}$ is completely described by the different.

\begin{cor}\label{degplem}
Let $L/K$ be as in Corollary~\ref{logdifcor} and assume that $[L:K]=p=\cha(\tilK)$. Then $\varphi_{L/K}$ has degrees 1 and $p$, and $\delta_{L/K}$ is the coefficient of the linear part of $\varphi_{L/K}$. In particular, the only break point is given by $r_0^{p-1}=\delta_{L/K}$.
\end{cor}


\subsection{Relation to the profile function}\label{concectsec}
We conclude the paper with describing the profile function in terms of the Herbrand function at all points not of type 1. In addition, one can describe the limit behaviour of $\varphi_f$ at wildly ramified points and type 2 points using the logarithmic Herbrand functions of valuation fields of height 2. Since we have not established the higher ramification theory in the latter case, we will provide the description and only outline the argument.

\subsubsection{Comparison theorem}
We start with the comparison at points not of type 1.

\begin{theor}\label{comparth}
Assume that $f\:Y\to X$ is a generically \'etale morphism between nice compact curves. Then for any point $y\in Y^\hyp$ with $x=f(y)$, $L=\calH(y)$ and $K=\calH(x)$, the profile function $\varphi_y$ of $f$ at $y$ coincides with the Herbrand function $\varphi_{L/K}$.
\end{theor}
\begin{proof}
We will deduce this from Theorem~\ref{invth}. By assumption, $L/K$ is separable, hence it suffices to check that the Herbrand function $\varphi_{L/K}$ satisfies conditions (a), (b) and (d) of that theorem. This was done in Theorem~\ref{Herth}(i), Corollary~\ref{tamehercor} and Corollary~\ref{degplem}, respectively.

\end{proof}

\subsubsection{The radii of $N_{f,\ge d}$}\label{radii}
For the sake of completeness, let us explicitly express the radii of the sets $N_{f,\ge d}$ in terms of the Herbrand function.

\begin{theor}\label{radiith}
Assume that $f\:Y\to X$ is a generically \'etale morphism between nice compact curves and $\Gamma=(\Gamma_Y,\Gamma_X)$ is a radializing skeleton of $f$, then

(i) Each $N_{f,d}$ with $d\notin p^\bfN$ is a union of edges and vertices of $\Gamma_Y$. So, its closure is a finite subgraph of $\Gamma_Y$.

(ii) The radius $r_i$ of the radial set $N_{f,\ge p^i}$ is computed as follows: if $y\in\Gamma_Y^{(2)}$, $L=\calH(y)$ and $K=\calH(x)$, then $r_i(y)$ is the break point $r$ of $\varphi_{L/K}$ such that $\deg\varphi_{L/K}<p^i$ precisely on the interval $[0,r)$.
\end{theor}
\begin{proof}
By Theorem~\ref{comparth}, $\varphi_{L/K}$ coincides with the profile function $\varphi_y$, hence its degree (or logarithmic derivative) is equal to the profile of the multiplicity function by Lemma~\ref{multlem}. The assertion of (ii) follows in an obvious way. 

Recall that by Theorem~\ref{multth2} the multiplicity outside of $\Gamma_Y$ takes values in $p^\bfN$. The multiplicity is constant along the edges of $\Gamma_Y$ by \cite[Lemma~3.5.8(ii)]{CTT}. This implies (i).
\end{proof}

\subsubsection{The limit behaviour at type 1 points}\label{type1sec}
Assume, now, that $f$ is wildly ramified at $y\in Y$ and $x=f(y)$. By \cite[Theorem~4.6.4]{CTT}, $\delta_f$ has a zero at $y$ whose order is equal to the order of the log different $\delta_{y/x}^\rmlog$ of $\calO_y/\calO_x$. In particular, if $I$ is an interval starting at $y$ and $I=[0,r_1]$ is a radius parametrization induced by a parameter $t_y\in m_y\setminus m_y^2$ then $\delta_f(r)=cr^{\delta_{y/x}^\rmlog}$. In \cite{CTT} one also gets rid of $c$ by rescaling $t_y$, but we will need the following finer construction. Provide $K_y=\Frac(\calO_y)$ with the valuation $|\ |_y$ of height two composed of the discrete valuation of $K_y$ with uniformizer $t_y$ and the standard valuation on the residue field $\calO_y/m_y=k$, and define $K_x$ analogously. To any $a\in|K^\circ_y|_y$ one can associate a monomial function $\psi_y(a)$ on $I$ as follows: there exists $c\in k$ with $a=|ct_y^n|_y$ and $|c|$ is uniquely defined by this condition, so we take $\psi_y(a)$ to be the absolute value of the function $ct_y^n$ on $I$ (i.e. $\psi_y(a)(r)=|c|r^n$ for any $r\in I$). This construction extends to piecewise monomial functions as follows: if $\varphi\:|K^\circ_y|\to|K^\circ_x|$ is a piecewise monomial function with breaks at $s_i$ then $\psi_y(\varphi)\:I\times[0,1]\to I\times[0,1]$ is the piecewise monomial function with breaks at $\psi_y(s_i)$ and the same degrees as $\varphi$ on the intervals. It is easy to see that the correspondence $\varphi\mapsto\psi_y(\varphi)$ preserves composition of functions locally at $y$.

Now we can describe the limit behaviour at $y$. Let $\varphi_y$ be the {\em logarithmic} Herbrand function of $K_y/K_x$. Then there exists $r_0\in(0,r_1]$ such that $\psi_y(\varphi_y)$ coincides with $\varphi_f$ on the subinterval $(0,r_0)$ of $I$. As usual, the proof uses the splitting method: both functions are compatible with compositions of morphisms, hence passing to the Galois closure and using a $p$-Sylow subgroup in the decomposition group of $y$ one reduces to the cases of tame morphisms and morphisms of degree $p$. The first case is, as usual, trivial. In the second case, both functions have a single break point, and proving that they are equal reduces to proving that $\psi_y(\delta_{K_y/K_x}^\rmlog)=\delta_f$. This is done essentially by the same argument as used in the proof of \cite[Theorem~4.6.4]{CTT}, namely, both sides are expressed as $|ht_yt_x^{-1}|$, where $h=\frac{dt_x}{dt_y}$.

\subsubsection{The limit behaviour at type 2 points}\label{type5sec}
A branch of $Y$ at a type 2 point $y$ can be described by a type 5 point $z$, see \cite[Section~3.4]{CTT}. The field $\calH(z)$ is a valued field of height two. If $I=[y,y']$ is an interval in the direction of $z$ then there is a natural map $\psi_z$ that associates to elements of $|\calH(z)|$ monomial functions on $I$. Furthermore, $\psi_z$ extends to piecewise monomial functions and $\psi_z(\varphi_{\calH(z)/\calH(f(z))})$ coincides with $\varphi_f$ on a small enough neighborhood of $y$ in $I$. The arguments are the same as outlined in the previous section.

\begin{rem}
The results stated in the last two sections indicate that it is more natural to interpret $\varphi_f$ as the logarithmic Herbrand function. This is not essential when interpreting the single value of $\varphi_f$ at a point $y\in Y^\hyp$, but becomes visible in the study of asymptotic behaviour. The same observation holds true already for the different function, see \cite[Section~4.7.3 and Remark~4.7.4]{CTT}.
\end{rem}

\appendix

\bibliographystyle{amsalpha}
\bibliography{Metric_Uniformization}

\end{document}